\documentclass[12pt,twoside]{amsart}
\usepackage[margin=3cm]{geometry}
\usepackage[colorlinks=false]{hyperref}
\usepackage[english]{babel}
\usepackage{graphicx,titling}
\usepackage{float}
\usepackage{amsmath,amsfonts,amssymb,amsthm}
\usepackage{lipsum}
\usepackage[T1]{fontenc}
\usepackage{fourier}
\usepackage{color}
\usepackage[latin1]{inputenc}
\usepackage{esint}
\usepackage{caption}
\usepackage{ccicons}

\makeatletter
\def\blfootnote{\xdef\@thefnmark{}\@footnotetext}
\makeatother

\newcommand\ccnote{
    \blfootnote{\copyright\,\, Luis Silvestre}
    \blfootnote{\ccLogo\, \ccAttribution\,\, Licensed under a \href{https://creativecommons.org/licenses/by/4.0/}{Creative Commons Attribution License (CC-BY)}.}
}

\usepackage[export]{adjustbox}
\numberwithin{equation}{section}
\usepackage{setspace}

\renewcommand{\leq}{\leqslant}

\renewcommand{\geq}{\geqslant}
\renewcommand{\mathbb}{\varmathbb}
\usepackage{fancyhdr}
\newtheorem{theorem}{Theorem}[section]
\newtheorem{lemma}[theorem]{Lemma}
\newtheorem{corollary}[theorem]{Corollary}
\newtheorem{proposition}[theorem]{Proposition}
\newtheorem{definition}[theorem]{Definition}
\newtheorem{remark}[theorem]{Remark}
\usepackage{dsfont}

\newcommand{\R}{\mathbb R}

\newcommand{\one}{\mathds 1}
\newcommand{\eps}{\varepsilon}
\newcommand{\I}{\mathrm{I}}
\newcommand{\dd} {\; \mathrm{d}}
\newcommand{\dvxt} {\; \mathrm{d}v\mathrm{d}x\mathrm{d}t}

\DeclareMathOperator*{\osc}{osc}
\DeclareMathOperator{\supp}{supp}
\DeclareMathOperator{\dist}{dist}

\DeclareMathOperator{\esssup}{ess \, sup}
\DeclareMathOperator{\essinf}{ess \, inf}

\newcommand{\trp}{(\partial_t + v \cdot \nabla_x)}

\newcommand{\HH}{\mathcal H}
\newcommand{\OO}{\mathcal O}
\newcommand{\Hk}{H^1_{kin}}
\newcommand{\z}{z^0}
\newcommand{\tz}{t^0}
\newcommand{\xz}{x^0}
\newcommand{\vz}{v^0}

\address{Luis Silvestre, University of Chicago, Department of Mathematics.}
\email{luis@math.uchicago.edu}

\begin{document}

\thispagestyle{empty}

\begin{minipage}{0.28\textwidth}
\begin{figure}[H]
%\centering
\includegraphics[width=2.5cm,height=2.5cm,left]{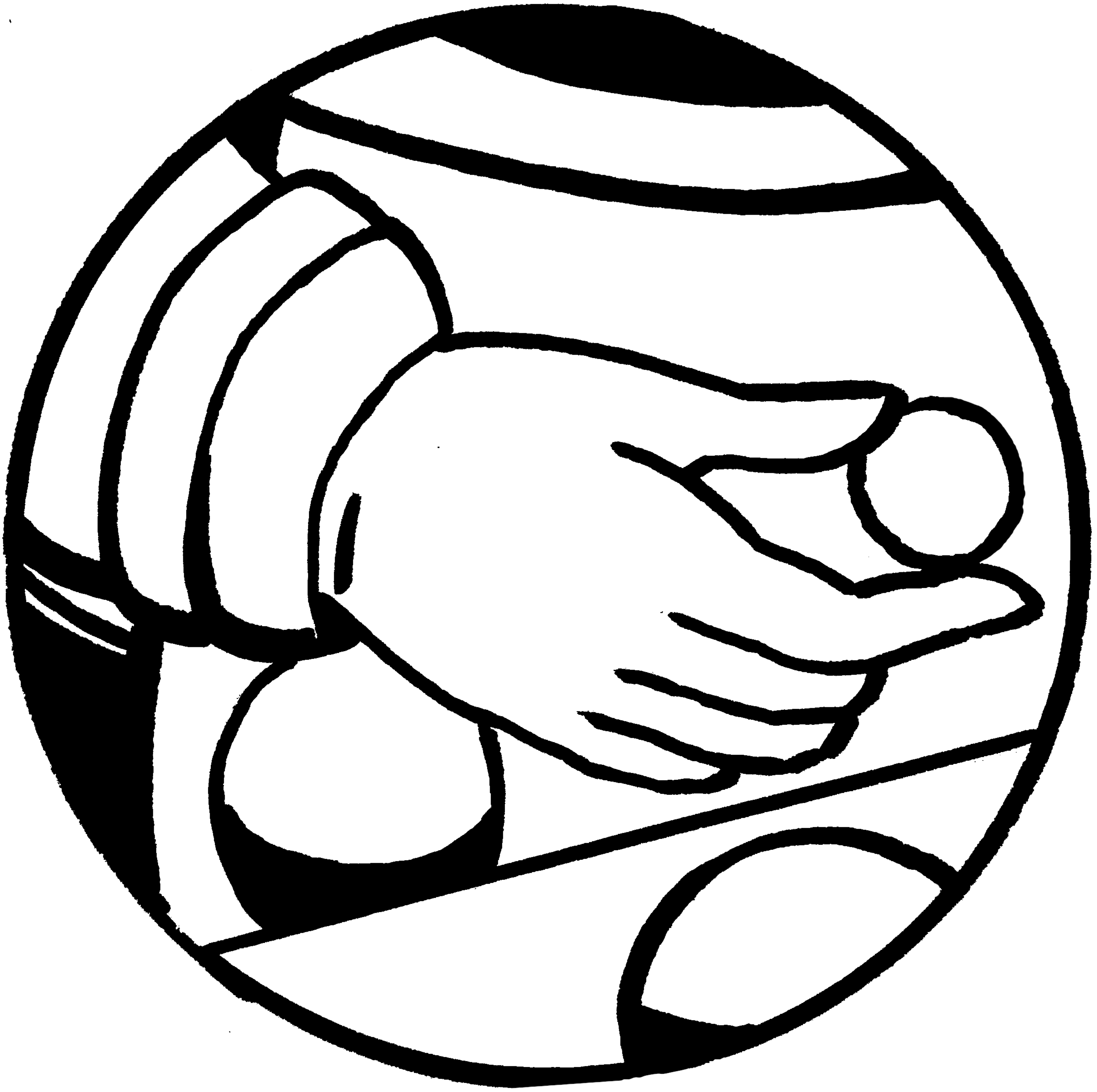}
\end{figure}
\end{minipage}
\begin{minipage}{0.7\textwidth} 
\begin{flushright}
%% The following metadata, in particular
%% the Paper No. and the DOI will be inserted by the journal
Ars Inveniendi Analytica (2022), Paper No. 6, 29 pp.
\\
DOI 10.15781/nqdd-qs03
\\
ISSN: 2769-8505
\end{flushright}
\end{minipage}

\ccnote

\vspace{1cm}

%%      -------------------------------------------------------------------------------
%%      -------------------------- TITLE ----------------------------
%%      -------------------------------------------------------------------------------
%% Authors, please put here the full title of the article

\begin{center}
\begin{huge}
\textit{H\"older estimates for kinetic Fokker-Planck equations up to the boundary}

\end{huge}
\end{center}

\vspace{1cm}

%%      -------------------------------------------------------------------------------
%%      -------------------------- AUTHORS AND AFFILIATIONS ----------------------------
%%      -------------------------------------------------------------------------------
%% Authors, please put here your full names and affiliations

\begin{center}
\begin{minipage}[t]{.28\textwidth}
\begin{center}
{\large{\bf{Luis Silvestre}}} \\
\vskip0.15cm
\footnotesize{University of Chicago}
\end{center}
\end{minipage}
\end{center}

\vspace{1cm}

%%% Please replace "James Mustard" below 
%%% with the name of the managing editor for your submission.
%%% If you are unsure about their identity
%%% please ask an editor-in-chief about.

\begin{center}
\noindent \em{Communicated by Scott Armstrong}
\end{center}
\vspace{1cm}

%%      -------------------------------------------------------------------------------
%%      -------------------------- BEGIN ABSTRACT ----------------------------
%%      -------------------------------------------------------------------------------
%% Authors, please put here the ABSTRACT and KEYBOARDS

\noindent \textbf{Abstract.} \textit{We obtain local H\"older continuity estimates up to the boundary for a kinetic Fokker-Planck equation with rough coefficients, with the prescribed influx boundary condition. Our result extends some recent developments that incorporate De Giorgi methods to kinetic Fokker-Planck equations. We also obtain higher order asymptotic estimates near the incoming part of the boundary. In particular, when the equation has a zero boundary conditions and no source term, we prove that the solution vanishes at infinite order on the incoming part of the boundary.}
\vskip0.3cm

\noindent \textbf{Kinetic equations.} Boundary estimates. H\"older regularity.
\vspace{0.5cm}

%%      -------------------------------------------------------------------------------
%%      -------------------------- BEGIN ARTICLE ----------------------------
%%      -------------------------------------------------------------------------------
%% Authors, copy the body of your paper here

\section{Introduction}

The purpose of this paper is to analyze the regularity near the spatial boundary for solutions of kinetic Fokker-Planck equations of the form
\begin{align}
(\partial_t + v\cdot \nabla_x) f - \partial_{v_i} (a_{ij}(t,x,v) \partial_{v_j} f) &= - b_i \partial_{v_i} f + G,  \text{ for } (t,x,v) \in [0,T] \times \Omega \times \R^d, \label{e:kFP} \\
f_{|\gamma_-} &= g. \label{e:influx}
\end{align}
It is a variant of Kolmogorov equation with rough diffusion coefficients. Here, the function $f=f(t,x,v)$ is defined in a domain of the form $[T_1,T_2] \times \Omega \times \R^d$, where $\Omega \subset \R^d$ is some open set with a $C^{1,1}$ boundary. We write $\gamma_-$ to denote the incoming part of the boundary. 

We prescribe the values of the solution $f$ on the incoming part of the boundary as $f=g$ on $\gamma_-$, for some arbitrary given function $g$. This is a common boundary condition for kinetic equations usually called \emph{influx}. It is one of the common physically relevant boundary conditions for kinetic equations. Other common choices are specular reflection and diffuse boundary conditions. The specular reflection boundary condition was essentially analyzed in \cite{guo2020}. We discuss it briefly in Section \ref{s:flattening}.

We say that the coefficients $a_{ij}(t,x,v)$ are uniformly elliptic, with ellipticity parameters $\Lambda$ and $\lambda>0$, if $\lambda \I \leq a_{ij}(t,x,v) \leq \Lambda \I$, for every $(t,x,v)$ in the domain of the equation.

Throughout this paper, we assume that the coefficients $a_{ij}$ are uniformly elliptic, and the drift vector field $b$ and the right-hand side $G$ are in $L^\infty$. 

In recent years, there have been exciting new developments obtaining interior H\"older estimates for kinetic Fokker-Planck equations like \eqref{e:kFP} (see \cite{polidoro2004,zhang2008,wangzhang2009,zhang2011,golse2019harnack,guerand2021log,guerand2021quantitative,anceschi2021note}). Our objective is to extend the H\"older continuity estimates of these results as estimates up to the boundary, in the case of the influx boundary conditions. We provide clean local regularity estimates that hold in sets that may contain parts of the boundary. When the domain $\Omega$ is convex, our estimates are stronger and essentially independent of the shape of $\Omega$.

Regularity results for Kinetic Fokker-Planck equations with variable coefficients have applications to the study of the Landau equation. For example, they play a key role in \cite{henderson2020c} to obtain conditional regularity estimates for the Landau equation with periodic boundary conditions in space. We hope that the results in this paper may help obtain similar regularity results for the Landau equation in bounded domains.

The interior H\"older continuity estimate for kinetic integro-differential equations was developed in \cite{imbert2020weak}. It plays a key role in the program of conditional regularity for the Boltzmann equation (see \cite{imbert2020ems,imbert2022jams}). It would be conceivable that the interior estimates of \cite{imbert2020weak} may be extended up to the boundary for the influx condition, similarly as we do it in this paper for the second order case. We have not pursued that direction yet. 

\subsection{Notation}

We write $\circ$ for the Galilean group operation.
\[ (t_1,x_1,v_1) \circ (t_2,x_2,v_2) = (t_1+t_2, x_1+x_2+t_2 v_1, v_1+v_2).\]
We use $Q_r$ to denote the kinetic cylinders.
\begin{align*}
Q_r &:= (-r^2,0] \times B_{r^3} \times B_r, \\
Q_r(z) &:= z \circ Q_r.
\end{align*}
We use the notation $\HH_r(z)$ for the points in $Q_r(z)$ that are inside the domain $\Omega$.
\[ \HH_r(z) := \{ (t,x,v) \in Q_r(z) : x \in \Omega \}. \]
%In the same way that $Q_r(z) = z \circ Q_r$, we write $\OO(z)$ for the set such that $\HH(z) = z \circ \OO(z)$. In other words
%\[ \OO((\tz,\xz,\vz)) := \{ (t,x,v) \in Q_r : x+t \vz + \xz \in \Omega \}.\]
Our definition for $\HH_r(z)$ depends on $\Omega$, even though it is not shown explicitly by this choice of notation. Observe that $\HH_r(z)$ is a convex subset of $\R^{1+2d}$ when $\Omega$ is a convex set in $\R^d$.

We write $\gamma$ for the points $(t,x,v) \in \R^{1+2d}$ so that $x \in \partial \Omega$. We differentiate three parts of this boundary depending on the direction of the flow: $\gamma = \gamma_- \cup \gamma_0 \cup \gamma_+$. Let $n=n(x)$ be the exterior unit normal vector to $\partial \Omega$ at the point $x$. We write
\begin{align*}
\gamma_- = \{ (t,x,v) \in \R^{1+2d} : x \in \partial \Omega \text{ and } v \cdot n < 0\},\\
\gamma_+ = \{ (t,x,v) \in \R^{1+2d} : x \in \partial \Omega \text{ and } v \cdot n > 0\},\\
\gamma_0 = \{ (t,x,v) \in \R^{1+2d} : x \in \partial \Omega \text{ and } v \cdot n = 0\}.
\end{align*}

\subsection{Main results}

The first main result is a local $L^\infty$ estimate up to the boundary. Its proof is rather short once we interpret the notion of subsolution correctly.

\begin{theorem} \label{t:Linfty}
Let $\z = (\tz,\xz,\vz) \in \R^{1+2d}$ with $\xz \in \overline \Omega$. Assume that $f$ is a nonnegative subsolution of \eqref{e:kFP} with $f=0$ on $\gamma_-$ (in the sense described in Definition \ref{d:weak-sub-sol} with $g=0$).  Assume that $G \in L^\infty(\HH_1(\z))$. Then
\[ \esssup_{\HH_{1/2}(\z)} f \leq C \left( \|f\|_{L^2(\HH_1(\z))} + \|G\|_{L^\infty(\HH_1(\z))}  \right).\]
The constant $C$ depends on dimension and the parameters of the equation, but it does not depend on the domain $\Omega$ or the point $\z$.
\end{theorem}

The notion of weak solution and subsolution will be made explicit in Definitions \ref{d:weak-sol} and \ref{d:weak-sub-sol}. While the definition of solution that we work with is arguably the most natural (and weakest), one could claim that our definition of nonnegative weak solution is more artificial. Theorem \ref{t:Linfty} will typically be applied to the positive or negative part of a solution. We prove in Lemma \ref{c:convex-sub-no-influx} that the positive part of a solution that is negative on $\gamma_-$ is a subsolution that vanishes on $\gamma_-$.

Our second main theorem is the H\"older continuity up to the boundary. Note that the estimate is better when $\Omega$ is convex. In this result, the estimate holds uniformly along the whole boundary $\gamma_+$, $\gamma_-$ and $\gamma_0$.

\begin{theorem} \label{t:Calpha}
There exists an $\alpha>0$ depending on the ellipticity parameters and dimension only so that the following statement holds.

Let $\z = (\tz,\xz,\vz) \in \R^{1+2d}$ with $x_0 \in  \overline\Omega$. Assume that $f$ is a solution of \eqref{e:kFP} (in the sense described in Definition \ref{d:weak-sol}). Then
\[ \|f\|_{C^\alpha(\HH_{1/2}(z^0))} \leq C \left( \|f\|_{L^2(\HH_1(z^0))} + \|g\|_{C^\alpha(\gamma_- \cap Q_1(\z))} + \|G\|_{L^\infty(\HH_1(\z)} \right). \]
The constant $C$ depends on dimension, the ellipticity parameters, $\|b\|_{L^\infty}$, $\Omega$ and $z^0$. However, when $\Omega$ is convex, $C$ and $\alpha$ depend on dimension and the parameters of the equation only (they do not depend on $z^0$ or the domain $\Omega$).
\end{theorem}

The $C^\alpha$ norm in Theorem \ref{t:Calpha} should be understood in terms of the kinetic distance defined in Section \ref{s:kinetic-distance}. If we used the usual Euclidean distance instead, we would still get some H\"older continuity (with a smaller exponent $\alpha$), but the constant $C$ would also depend on the velocity $\vz$.

When applying local regularity estimates to study the global regularity of solutions to the kinetic equations (for example as in \cite{henderson2020c} or \cite{imbert2022jams}), we need to have some control on how these estimates behave for large velocities. To that effect, it is optimal to have local estimates like those for convex domains $\Omega$, where the constants involved are independent of the point $\vz$.

When $\Omega$ is not convex, the pure transport equation $\trp f = 0$, with the influx condition, generates discontinuities. Theorem \ref{t:Calpha} tells us that once we add diffusion in $v$, these discontinuities disappear and the solutions are always $C^{\alpha}$. However, for large velocities, the transport terms are stronger and it is natural to expect that the constant $C \to \infty$ as the tangential component of $\vz$ diverges in some direction where $\Omega$ has a negative curvature.

Our last main result concerns the incoming part of the boundary $\gamma_-$ only. It says that when the equation has a zero source term, and a zero boundary condition, then the solution vanishes to infinite order on $\gamma_-$.

\begin{theorem} \label{t:Cinfty}
Let $\z \in \gamma_-$. Assume that $f$ is a solution of \eqref{e:kFP} (in the sense described in Definition \ref{d:weak-sol}) with $g=0$ and $G=0$. Then, for every exponent $q \in \mathbb N$, there exists an constant $C$ so that for all $r>0$,
\[ \esssup_{\HH_{r}(\z)} f \leq C \|f\|_{L^2(\HH_1(\z))} r^q. \]
The constant $C$ depends on $q$, dimension, the parameters of the equation, $\Omega$ and $\z$. However, when $\Omega$ is convex, it depends on $q$, dimension, the parameters of the equation, and $\vz \cdot n$ only.
\end{theorem}

When there is a nonzero boundary value $g$, or a nonzero source term $G$, a version of Theorem \ref{t:Cinfty} holds for a restricted range of exponents $q$. If we have a boundary value $g$ that is only H\"older continuous with exponent $\alpha$, the exponent $q$ cannot be taken larger than $\alpha$. If we have a bounded source term $G \in L^\infty$, the exponent $q$ cannot be taken larger than $2$. We explore these and other variants of Theorem \ref{t:Cinfty} by the end of section \ref{s:Cinfty}.

It is not common to have an infinite order of vanishing near the boundary of a partial differential equation. Theorem \ref{t:Cinfty} is quite unusual in this respect. If we consider for example the Laplace equation in a smooth domain with a Dirichlet boundary condition, having the boundary value equal to zero in some open set would not imply that the solution vanishes at infinite order there. A solution to the heat equation starting from an initial data that equals zero in some open set will vanish at infinite order at its initial time. This is also true for general parabolic equations with measurable (uniformly elliptic) coefficients and follows from the well known Gaussian upper bounds for their fundamental solutions.\footnote{We thank Chris Henderson and an anonymous referee for pointing this out.} Our setting in this paper is comparable to the latter since the diffusion in \eqref{e:kFP} is parallel to the boundary while the drift is transversal. However, the proof we provide for Theorem \ref{t:Cinfty} follows a completely independent path based on the geometry of kinetic cylinders.

\medskip

When we were writing this paper, we learned about the very recent work \cite{zhu2022regularity}. In that paper Yuzhe Zhu analyzes the well posedness and boundary regularity for a Fokker-Planck equation like \eqref{e:kFP} with the three most common boundary conditions: prescribed influx, specular reflection, and diffuse. While the estimates provided in the main theorem of \cite{zhu2022regularity} are global, one can also find local estimates in the body of the paper (see Remark 1.2 and Proposition 3.7 in \cite{zhu2022regularity}) that are similar to our Theorem \ref{t:Calpha} in the non-convex case. In \cite{zhu2022regularity}, there is no attention to the convexity of the domain $\Omega$. The domain is initially flattened with the same change of variables we describe in Section \ref{s:flattening}. All the local estimates obtained in \cite{zhu2022regularity} depend on $\vz$ and the curvature of the boundary. There is no special analysis focusing near $\gamma_-$, like in our Theorem \ref{t:Cinfty}. The main theorems in \cite{zhu2022regularity} contain other issues that are outside the scope of this paper, like the existence of solutions, and global estimates for diffuse boundary conditions.

In the appendix of \cite{zhu2022regularity} there is an example showing that solutions to the equation \eqref{e:kFP} are not in general differentiable on $\gamma_0$, even when the diffusion coefficients $a_{ij}$ are constant and the solution is stationary. This example shows that a naive higher order regularity estimate (as in Schauder estimates) may not hold in this setting.

\subsection{Additional notation}
When we integrate a quantity over some portion of $\gamma$, we write $\dd \gamma$ to denote $\dd v \dd S(x) \dd t$. Here, $\dd S(x)$ is the differential of surface for $x \in \partial \Omega$.

For any open subset $D \subset \R^d$, we sometimes write $L^2_{t,x} H^1_v$ or $L^2_{t,x} H^{-1}_v$ to denote the Sobolev spaces and their norms taking into account derivatives with respect to the velocity variable only. That is,
\begin{align*}
\|f\|_{L^2_{t,x} H^1_v(D)}^2 &:= \|f\|_{L^2(D)}^2 + \|\nabla_v f\|_{L^2(D)}^2,\\
\|g\|_{L^2_{t,x} H^{-1}_v(D)}^2 &:= \sup \left\{ \int_D g \varphi \dvxt : \varphi \in C^1_c(D) \text{ with } \|\varphi\|_{L^2_{t,x} H^1_v(D)} \leq 1 \right\}.
\end{align*}

\subsection*{Acknowledgment}

Luis Silvestre is supported by NSF grants 2054888 and 1764285.

\section{Galilean invariance and convolutions}

While the class of equations \eqref{e:kFP} is invariant by translations in time and space, we cannot translate a solution $f$ in velocity and expect it to solve an equation of the same kind. Indeed, the \emph{transport} term $(\partial_t + v \cdot \nabla_x) f$ has a coefficient that depends on $v$. The correct group of transformations to associate with this class of equations is the group of inertial changes of variables. We define the following group operator in $\R^{1+2d}$.
\[ (s,y,w) \circ (t,x,v) = (s+t, x+y+tw, v+w).\]
The operator $\circ$ defines a non-commutative Lie group structure in $\R^{1+2d}$ sometimes called \emph{Galilean group}.

We may reconsider the differential operators $(\partial_t + v \cdot \nabla_x)$, $\nabla_v$ and $\nabla_x$ in terms of this operation. They turn out to be the natural differential operators that arise from the right action of the group.
\begin{align*}
(\partial_t + v \cdot \nabla_x) f(t,x,v) &= \lim_{h \to 0} \frac{f((t,x,v)\circ(h,0,0)) - f(t,x,v)} h , \\
\partial_{v_i} f(t,x,v) &= \lim_{h \to 0} \frac{f((t,x,v)\circ(0,0,h e_i)) - f(t,x,v)} h , \\
\partial_{x_i} f(t,x,v) &= \lim_{h \to 0} \frac{f((t,x,v)\circ(0,h e_i,0)) - f(t,x,v)} h.
\end{align*}

We can immediately verify that these differential operators are \emph{left}-invariant by the action of the group: if $z_0 \in \R^{1+2d}$ and we define $\tau_{z_0} f(z) = f(z_0 \circ z)$, then $\tau_{z_0} (\partial_t + v \cdot \nabla_x) f = (\partial_t + v \cdot \nabla_x) \tau_{z_0} f$, $\tau_{z_0} \nabla_x f = \nabla_x \tau_{z_0} f$, and $\tau_{z_0} \nabla_v f = \nabla_v \tau_{z_0} f$.

The differential operators arising from the left action of the group are right-invariant, but not necessarily left-invariant. They are
\begin{align*}
\partial_t f(t,x,v) &= \lim_{h \to 0} \frac{f((h,0,0) \circ (t,x,v)) - f(t,x,v)} h , \\
(\partial_{v_i} + t \partial_{x_i}) f(t,x,v) &= \lim_{h \to 0} \frac{f((0,0,h e_i) \circ (t,x,v)) - f(t,x,v)} h , \\
\partial_{x_i} f(t,x,v) &= \lim_{h \to 0} \frac{f((0,h e_i,0) \circ (t,x,v)) - f(t,x,v)} h.
\end{align*}

The equation \eqref{e:kFP} involves the transport operator $(\partial_t + v \cdot \nabla_x)$ and derivatives in velocity $\partial_{v_i}$. The Galilean group, is the natural group of transformations in $\R^{1+2d}$ that is used throughout the paper.

We define the convolution of functions in terms of the Galilean group.
\[ f \ast g (z) := \int_{\R^{1+2s}} f(\omega) g(\omega^{-1} \circ z) \dd \omega. \]
This convolution is associative, but it is not commutative. If we make the change of variables $\omega^{-1} z \mapsto \omega$, we obtain the equivalent expression
\[ f \ast g (z) := \int_{\R^{1+2s}} f(z \circ \omega^{-1}) g(\omega) \dd \omega. \]
Throughout this paper, whenever we write a convolution, we mean this convolution with respect to the Galilean group (we basically give up the usual group structure of $\R^{1+2d}$ altogether in favor of the Galilean group). The convolution with respect to the Galilean group has very similar properties as the usual convolution, provided that we are careful as to whether operations apply from the left or from the right in every case.

By a direct computation, we can easily verify that the left-invariant differential operators can be thrown into the second factor:
\begin{align*}
\trp (g\ast f)(t,x,v) &=g \ast  (\partial_t + v \cdot \nabla_x) f (t,x,v),\\
\partial_{v_i} (g\ast f) (t,x,v) &= g \ast  \partial_{v_i} f (t,x,v), \\
\partial_{x_i} (g\ast f) (t,x,v) &= g \ast \partial_{x_i} f (t,x,v) .
\end{align*}
Conversely, the right-invariant differential operators can be transferred into the first factor:
\begin{align*}
\partial_t (g\ast f)(t,x,v) &=[\partial_t g] \ast f (t,x,v),\\
(\partial_{v_i} + t \partial_{x_i}) (g\ast f) (t,x,v) &= [(\partial_{v_i} + t \partial_{x_i}) g] \ast  f (t,x,v), \\
\partial_{x_i} (g\ast f) (t,x,v) &= [\partial_{x_i} g] \ast f (t,x,v) .
\end{align*}
Note that $f \ast g$ is $C^\infty$ provided that at least one of the two functions is $C^\infty$. Moreover
\[ \int (g \ast f)(z) \varphi(z) \dd z = \int f(z) (\hat g \ast \varphi)(z) \dd z,\]
where $\hat g(z) = g(z^{-1})$.

Convolving with an appropriately scaled family of mollifiers gives us a convenient smooth approximation for rough functions $f$. We use it in Section \ref{s:sobolev}, for technical manipulations of functions in kinetic Sobolev spaces.

%At some point in this paper we use Young's convolution inequality \luis{of weak type}, that we state here without a proof.

%\begin{proposition} \label{p:weak-young-ineq}
%Let $f \in L^p(\R^{1+2d})$ and $g \in L^{q,\infty}(\R^{1+2d})$. Assume $1 < p,r,q < \infty$, where $1/r + 1= 1/p + 1/q$. %Then $f\ast g \in L^r$ and $\|f \ast g\| \leq C \|f\|_{L^p} \|g\|_{L^{q,\infty}}$.
%\end{proposition}

Note that while the family of kinetic Fokker-Planck equations \eqref{e:kFP} is invariant by Galilean translations, the boundary condition \eqref{e:influx} is not. The spatial domain $\{x \in \Omega\}$ would become an oblique domain after a translation $z \mapsto \z \circ z$ if $\vz \neq 0$. So, the assumptions in Theorems \ref{t:Linfty}, \ref{t:Calpha} and \ref{t:Cinfty} cannot be easily reduced to $\z = 0$ by a translation. The parameters in the estimates may depend on $\vz$, and in fact the constant $C$ in Theorems \ref{t:Calpha} and \ref{t:Cinfty} degenerates for large velocities when $\Omega$ is not convex.

\subsection{Kinetic distance and H\"older norms}

\label{s:kinetic-distance}

The distance in $\R^{1+2d}$ should be appropriately adapted to be homogeneous with respect to the kinetic scaling $S_r$ and left invariant with respect to the left action of the Galilean group. We provide an explicit formula for such a distance following \cite{imbert2020schauder}.

\begin{definition}\label{d:distance}
Given two points $z_1 = (t_1,x_1,v_1)$ and $z_2 = (t_2, x_2, v_2)$ in $\R^{1+2d}$, we define the following distance function
\[ d_\ell(z_1,z_2) := \min_{w \in \R^d} \left\{ \max \left( |t_1-t_2|^{ \frac 1 {2} } , |x_1-x_2-(t_1-t_2)w|^{ \frac 1 {3} } , |v_1-w| , |v_2-w| \right) \right\}.\]
\end{definition}

The subindex ``$\ell$'' stands for \textbf{``l''}eft invariant. It is easy to check that this distance $d_\ell$ satisfies the following two invariances.
\begin{align*}
d_\ell(S_r z_1, S_r z_2) &= r d_\ell(z_1,z_2), \\
d_\ell(z \circ z_1, z\circ z_2) &= d_\ell(z_1,z_2).
\end{align*}

We define the H\"older spaces and their norms using this distance. For $\alpha \in (0,1)$, the (kinetic) $C^\alpha$ norm of a function $f:D \to \R$, for $D \subset \R^{1+2d}$ is given by
\[ [f]_{C^\alpha(D)} = \sup_{z_1 \neq z_2 \in D} \frac{|f(z_1) - f(z_2)|}{d_\ell(z_1,z_2)^\alpha}, \qquad  \|f\|_{C^\alpha(D)} = [f]_{C^\alpha(D)} + \|f\|_{L^\infty(D)}. \]

This is the H\"older norm in Theorem \ref{t:Calpha}, with respect to the kinetic distance $d_\ell$.

\section{Flattening the boundary}
\label{s:flattening}

In this section, following \cite{MR3592757} and \cite{guo2020}, we explain the change of variables to transform the equation \eqref{e:kFP} from an arbitrary domain to a flat boundary. We use it to prove Theorems \ref{t:Calpha} and \ref{t:Cinfty} only when $\Omega$ is not convex. We also use it for some technical lemmas about kinetic Sobolev spaces in Section \ref{s:sobolev}.

It is a local change of variables. Given a point $\xz \in \partial \Omega$, we build a transformation $\Phi : (B_r(\xz) \cap \Omega) \times \R^d \to \R^d_- \times \R^d$, and define $f(t,x,v) = \tilde f(t,\Phi(x,v))$, so that this new function $\tilde f$ satisfies an equation of the form \eqref{e:kFP} in $[0,T] \times (B_\rho(0) \cap \{ x_1 < 0\}) \times \R^d$. The change of variables $(y,w) = \Phi(x,v)$ involves both variables together. This is important to keep the transport part of the equation essentially unchanged.

We assume that the domain $\Omega$ has a $C^{1,1}$ boundary. Let $\phi: B_r(\xz) \to \R^d$ be a transformation that flattens the boundary (i.e. $\phi(\partial \Omega) \subset \{0\} \times \R^{d-1}$). Since $\partial \Omega$ is a Lipschitz boundary, we can make $D \phi$ and $D \phi^{-1}$ both bounded. Moreover, we assume that $\partial \Omega$ is $C^{1,1}$ so that $D^2 \phi$ is bounded as well.

Let us define $\Phi(x,v) = (\phi(x) , D\phi(x) v)$. Following the notation in \cite{guo2020}, we write $A = D\phi(x)$, $y = \phi(x)$, and $w = Av$. Note that the matrix $A=A(x)$ depends on $x$ (or $y$), but not on $v$ (or $w$, for $y$ fixed).

We define $\tilde f(t,y,w) = f(t,\Phi^{-1}(y,w))$. Equivalently, $f(t,x,v) = \tilde f(t,\Phi(x,v))$. The domain of the function $\tilde f$ is the image of $\Phi$, which contains a neighborhood of $\bar y = \phi(\xz)$ and all values of $w \in \R^d$.

The function $f$ satisfies the equation \eqref{e:kFP}. By a direct computation, we verify that the function $\tilde f$ satisfies the equation
\begin{eqnarray*}
 &&(\partial_t + w \cdot \nabla_y) \tilde f - \partial_{w_i} (\tilde a_{ij}(t,y,w) \partial_{w_j} \tilde f) = - \tilde b_i \partial_{w_i} \tilde f + \tilde G ,  
 \\
 &&\hspace{5cm}\text{ for } (t,y,w) \in [0,T] \times (B_\rho(\bar y) \cap \{y_1 < 0\}) \times \R^d,
\end{eqnarray*}
\begin{eqnarray*}
\mbox{where}&&\tilde G(t,y,w) = G(t,x,v), 
\\
%\tilde H_i(t,y,w) = A_{ij} H_j(t,x,v), \\
&&\tilde b_i(t,y,w) = A_{ij} b_j(t,x,v) + v_j v_k \frac{\partial^2 \phi_i}{\partial x_j \partial x_k}, 
\\
&&\tilde a_{ij}(t,y,w) = A_{ir} A_{js} a_{rs}(t,x,v).
\end{eqnarray*}

The matrices $A$ and $A^{-1}$ are bounded because $\partial \Omega$ is a Lipschitz boundary. We have $\|\tilde a_{ij}\| \leq \|A\|^2 \|a_{ij}\|$. Moreover, the smallest eigenvalue of $\tilde a_{ij}$ is larger than or equal to the smallest eigenvalue of $a_{ij}$ times $\|A^{-1}\|^{-2}$. Thus, the uniform ellipticity parameters of $\tilde a_{ij}$ depend on the uniform ellipticity parameters of $a_{ij}$ and the upper bounds for $\|A\|$ and $\|A^{-1}\|$. We conclude that the coefficients $\tilde a_{ij}$ are uniformly elliptic provided that the original coefficients $a_{ij}$ are, and $\partial \Omega$ is a Lipschitz boundary.

The $L^\infty$ bound for $\tilde G$ is the same as the $L^\infty$ norm of $G$. The drift $\tilde b$ has an extra term that involves second derivatives of $\phi$. We assume that $\partial \Omega$ is a $C^{1,1}$ boundary only to have a bound on this term.

We observe that this new function $\tilde f$ satisfies an equation that retains the same assumptions as the original one for $f$. However, the spatial domain now has a flat boundary. With this transformation, we reduce the study of local regularity estimates for the equation \eqref{e:kFP}, to the case of flat boundaries. In particular, proving any result like Theorems \ref{t:Linfty}, \ref{t:Calpha} and \ref{t:Cinfty} in the case of the flat boundary, would immediately imply the same result for $C^{1,1}$ boundaries. The only problem is that the drift term $\tilde b$ depends on $v$ and the curvatures of the boundary. Thus, we would obtain local estimates with constants depending on the domain $\Omega$ and the velocity nearby. For domains $\Omega$ that are convex, we do not use the change of variables and we get estimates that do not depend on the domain $\Omega$ or the velocity $\vz$ in Theorems \ref{t:Calpha} and \ref{t:Cinfty}. We use this change of variables to reduce the general case of both theorems to the convex case.

\subsection{Specular reflection boundary condition}

In \cite{guo2020}, the authors explain how to use the change of variables that flattens the boundary together with a mirror extension to study kinetic equations with specular reflection boundary condition. A similar procedure is carried out in \cite{zhu2022regularity}. We explain it briefly in this section. Let us recall the specular reflection boundary condition. It consists in postulating that the function $f$ satisfies the following identity for all $(t,x,v) \in \gamma$,
\begin{equation} \label{e:specular-reflection}
f(t,x,v) = f(t,x,Rv), \text{ where } Rv = v - 2(v \cdot n) n.
\end{equation}
Let us define our notion of weak solution.

\begin{definition} \label{d:weak-sol-specular}
Let $D$ be a subset of $\R \times \Omega \times \R^d$. Assume that whenever $(t,x,v) \in \gamma \cap \partial D$, then also $(t,x,Rv) \in \gamma \cap \partial D$. We say that a function $f \in L^2(D)$ is a weak solution of \eqref{e:kFP} in $D$ with the specular reflection boundary condition \eqref{e:specular-reflection} if 
\begin{itemize}
\item $\nabla_v f \in L^2(D)$ 
\item For any test function $\varphi \in C^1_c(\R^{1+2d})$ so that $\supp \varphi \cap \partial D \subset \gamma$, and $\varphi(t,x,v) = \varphi(t,x,Rv)$ for any $(t,x,v) \in \gamma$, we have
\[ \begin{aligned}
\int_D - f (\partial_t + v \cdot \nabla_x) \varphi + a_{ij} \partial_{v_j} f \partial_{v_i} \varphi + b_i \partial_{v_i} f \varphi - G \varphi \dvxt = 0.
\end{aligned}
\]
\end{itemize}
\end{definition}

The mirror extension consists of the following technique. Assume that $\Omega = \{x_1 < 0\}$. We can always perform a change of variables as in Section \ref{s:flattening} to reduce to this case. We have $\gamma = \{x_1=0\}$. We extend the function $f$, and the equation, to $\{x_1>0\}$ in the following way. We write $R(v_1,v') = (-v_1,v')$ for $v' \in \R^{d-1}$. We also apply the reflection operator to the space variable $R(x_1,x') = (-x_1,x')$. We write $\tilde D$ to be the domain $D$ extended to the other side of $\gamma$ by mirror reflection
\[ \tilde D = D \cup (\partial D \cap \gamma) \cup \{(t,x,v): (t,Rx,Rv) \in D\}.\]
We extend the function $f$ to the whole domain $\tilde D$:
\[ \tilde f(z) = \begin{cases}
f(t,x,v) & \text{ if } (t,x,v) \in D, \\
f(t,Rx,Rv) & \text{ if } (t,x,v) \in \{(t,x,v): (t,Rx,Rv) \in D\}
\end{cases}
\]
Note that $\gamma$ has measure zero, so it is ok not to specify the value of $\tilde f$ there. A posteriori, $\tilde f$ will be extended on $\partial D \cap \gamma$ by continuity. We also extend the coefficients $a_{ij}$, the source term $G$, and the drift $b$, to $\tilde D$:
\begin{align*}
\tilde a_{ij}(t,x,v) &= \begin{cases}
a_{ij}(t,x,v) & \text{ if } (t,x,v) \in D, \\
a_{ij}(t,Rx,Rv) & \text{ if } (t,x,v) \in \{(t,x,v): (t,Rx,Rv) \in D\}
\end{cases} \\
\tilde G(t,x,v) &= \begin{cases}
G(t,x,v) & \text{ if } (t,x,v) \in D, \\
G(t,Rx,Rv) & \text{ if } (t,x,v) \in \{(t,x,v): (t,Rx,Rv) \in D\}
\end{cases} \\
\tilde b(t,x,v) &= \begin{cases}
b(t,x,v) & \text{ if } (t,x,v) \in D, \\
Rb(t,Rx,Rv) & \text{ if } (t,x,v) \in \{(t,x,v): (t,Rx,Rv) \in D\}
\end{cases} \\
\end{align*}

The following proposition is proved in \cite{guo2020}.
\begin{proposition} \label{p:mirror-extension}
If $f$ solves \eqref{e:kFP} with the specular reflection boundary condition \eqref{e:specular-reflection} in the domain $D$, then $\tilde f$ solves the same equation, with the coefficients $\tilde a_{ij}$, source term $\tilde G$, and the drift $\tilde b$, in the domain $\tilde D$ (across $\gamma$).
\end{proposition}

Thanks to Proposition \ref{p:mirror-extension}, the analysis of local regularity estimates for the equation \eqref{e:kFP} with the specular reflection boundary condition \ref{e:specular-reflection} is reduced to local interior estimates for the equation \eqref{e:kFP}. In particular, the following result holds.

\begin{theorem} \label{t:specular}
Let $\z = (\tz,\xz,\vz) \in \R^{1+2d}$ with $x_0 \in \Omega$. Assume that $f$ is a solution of \eqref{e:kFP} with the specular reflection boundary condition \eqref{e:specular-reflection}. The following $C^\alpha$ estimate holds.
\[ \|f\|_{C^\alpha(\HH_{1/2}(z^0))} \leq C \left( \|\tilde f\|_{L^2(Q_1(z^0))} + \|\tilde G\|_{L^\infty(Q_1(\z))} \right). \]
The constants $\alpha>0$ and $C$ depend on dimension, the ellipticity parameters, $\|\tilde b\|_{L^\infty}$, $\Omega$ and $z^0$.
\end{theorem}

Note that the right-hand side of the inequality involves the norms of the extended functions $\tilde f$ and $\tilde G$ in $Q_1(\z)$. The values of these functions depend on the values of $f$ and $G$ at points $(t,x,v)$ so that $(t,Rx,Rv) \in Q_1(\z)$. Depending on the value of $\z$, these points may be outside of $\HH_1(\z)$, and potentially quite far away from $\z$ with respect to the kinetic distance. Also, since the mirror extension relies on flattening the boundary first, the constant $C$ may depend on the velocity $\vz$. As we explained in Section \ref{s:flattening}, the change of variables that flattens the boundary introduces an extra drift term that becomes large for large values of $|v|$.

\section{Kinetic Sobolev spaces}
\label{s:sobolev}

Following \cite{albritton2019variational}, we define the kinetic Sobolev space as follows.

\begin{definition}
Given an open set $D \subset \R^{1+2d}$, we say $f \in \Hk(D)$ if $f \in L^2(D)$, $\nabla_v f \in L^2(D)$, and $(\partial_t + v \cdot \nabla_x) f \in L^2_{t,x} H^{-1}_v$ in the sense that
\[ \int f(t,x,v) (\partial_t + v \cdot \nabla_x) \varphi \dvxt \leq C \| \nabla_v \varphi \|_{L^2(D)}, \]
for any function $\varphi \in C^1_c(D)$. We further write
\[ \|f\|_{\Hk(D)}^2 = \|f\|_{L_2(D)}^2 + \|\nabla_v f\|_{L^2(D)}^2 + \|\trp f\|_{L^2_{t,x} H^{-1}_v(D)}^2 .\]
\end{definition}

The kinetic Sobolev spaces $\Hk$ are introduced in \cite{albritton2019variational} with a Gaussian weight with respect to the velocity variable. In the context of this paper, it is not convenient to consider any weight. Many of the basic properties of the kinetic space proved in \cite{albritton2019variational} work without a weight, with even a slightly cleaner proof.

Using convolutions with respect to the Galilean group is a convenient way to approximate functions in $\Hk$ with smooth ones. Let us consider a compactly-supported smooth function $\eta: \R^{1+2d} \to \R$ with integral one. Let us use kinetic scaling to produce a family of mollifiers:
\[ \eta_\eps(t,x,v) = \eps^{-2-4d} \eta(\eps^{-2} t, \eps^{-3}x, \eps^{-1} v). \]
If $\nabla_v f \in L^2(\R^{1+2d})$, we observe that $\nabla_v (\eta_\eps \ast f) = \eta_\eps \ast \nabla_v f$ converges to $\nabla_v f$ in $L^2$ and almost everywhere. Likewise, if $(\partial_t + v \cdot \nabla_x) f \in L^2_{t,x}(\R^{1+d},H_v^{-1}(\R^d))$, then $(\partial_t + v \cdot \nabla_x) (\eta_\eps \ast f) = \eta_\eps \ast (\partial_t + v \cdot \nabla_x) f$ converges to $(\partial_t + v \cdot \nabla_x) f$ in that same space. This last statement would not be true if we were using the convolution with respect to the usual Euclidean structure of $\R^{1+2d}$.

The next lemma is less general than a similar result in \cite{albritton2019variational}. Here, we approximate an arbitrary function $f \in \Hk(\HH_1(\z))$ with smooth functions. Interestingly, we can construct the smooth approximations $f_\eps \to f$ up to the boundary with a simple mollification.
\begin{lemma} \label{l:smooth-approx}
Let $f \in \Hk(\HH_1(\z))$. For $\eps >0$, there exists a family of smooth functions $f_\eps : \overline{\HH_{1-\eps}(\z)} \to \R$, so that $f_\eps \to f$ in $\Hk(\HH_{1-\eps_0})$ as $\eps \to 0$ (for any fixed $\eps_0>0$).
\end{lemma}

\begin{proof}
Using the change of variables described in Section \ref{s:flattening}, we assume without loss of generality that the boundary $\partial \Omega$ is flat: $\Omega = \{ x_1 < 0\}$. Moreover, by a simple translation (using the Galilean group structure), we also assume that $\z = (0,0,(v_1^0,0,\dots,0))$. We cannot assume that $v_1^0 = 0$ because a Galilean change of variables that modifies the velocity component that is normal to the boundary $\partial \{x_1 > 0\}$ would also modify the boundary of the equation. It is straight forward to verify that local norms of $f$ in $\Hk$ and $\tilde f$ in $\Hk$ are comparable.

Since $f \in \Hk(\HH_1)$, in particular $(\partial_t + v \cdot \nabla_x) f \in L^2_{t,x} H^{-1}_v$. Thus, there exists some vector field $F: \HH_1 \to \R^d$ such that $\|F\|_{L^2(\HH_1)} \lesssim \|f\|_{\Hk}$ and 
\[ (\partial_t + v \cdot \nabla_x) f = \partial_{v_i} F_i.\]

Let us define $\bar f$ and $\bar F$ by extending $f$ and $F$ to all of $Q_1(z^0)$ making them equal to zero when $x_1 > 0$.

Consider a smooth function $\eta_1$, compactly supported and with unit integral. Let us take a function $\eta_1$ that is supported in the set $\{x_1>0\} \cap \{v_1<0\} \cap \{t > 0\}$. We scale it to turn it into an approximation of the identity.
\[ \eta_r(t,x,v) = r^{-2-4d} \eta_1(r^{-2}t, r^{-3}x, r^{-1} v).\]
We define 
\[ f_\eps := \eta_\eps \ast \bar f.\]
Clearly, we have $f_\eps \in C^\infty$ and $\nabla_v f_\eps \to \nabla_v f$ in $L^2(\HH_{1-\eps_0})$, for any $\eps_0 > 0$. The convergence of $(\partial_t + v \cdot \nabla_x) f_\eps$ in $L^2_{t,x} H^{-1}_v$ results from the choice of the support of $\eta_\eps$. Typically, when we extend a function $f$ as zero, we may be creating a singular part for its derivatives across the boundary. The choice of the support of $\eta_1$ is so that $f_\eps(z)$ depends only on the values of $f$ in $\HH_1$, for any $z \in \HH_{1-\eps_0}$. Let $\varphi$ be a $C^1$ function whose support is inside $\HH_{1-\eps_0}$. We compute
\begin{align*}
&\int \trp f_\eps \varphi \dvxt = - \int f_\eps \trp \varphi \dvxt, \\
&= -\int (\eta_\eps \ast \bar f) \trp \varphi \dvxt= -\int -\bar f \trp (\hat \eta_\eps \ast \varphi) \dvxt.
\end{align*}
Noticing that $\hat \eta_\eps \ast \varphi$ is supported inside $\HH_1$ for $\eps$ small enough, we find %%% This is something I should double or triple check.
\begin{align*}
\int \trp f_\eps \varphi \dvxt &= \int_{\HH_1} \trp f (\hat \eta_\eps \ast \varphi) \dvxt, \\
&= \int_{\HH_1} \partial_{v_i} F_i (\hat \eta_\eps \ast \varphi) \dvxt, \\
&= \int_{\HH_1} \partial_{v_i} (\eta_\eps \ast F_i) \varphi \dvxt.
\end{align*}
Thus, $\trp f_\eps = \partial_{v_i} (\eta_\eps \ast F_i)$ in $\HH_{1-\eps_0}$. But clearly $\eta_\eps \ast F_i \to F_i$ in $L^2$, from which we deduce that $\trp f_\eps \to \trp f$ in $L^2_{t,x} H^{-1}_v$.
\end{proof}

The following proposition allows us to define some form of boundary values for a function $f$ in $\Hk$. It is related to but not as strong as the conjectured Question 1.8 in \cite{albritton2019variational}.

\begin{proposition} \label{p:trace}
The restriction operator $f \mapsto f_{|\gamma}$ is well defined from $\Hk(\HH_1(\z))$ to $L^2_{loc}(\gamma,\omega)$, for the weight $\omega = \min(|v \cdot n| , (v \cdot n)^2)$. More precisely, for any $\tilde \gamma$ that is compactly contained in $\gamma \cap \HH_1(\z)$, there is a constant $C$ (independent of $\z$) so that for all $f \in \Hk(\HH_1(\z))$, then
\[ \int_{\tilde \gamma} f^2 \omega \dd \gamma \leq C \|f\|^2_{\Hk(\HH_1(\z))}.\]

Moreover, if $f_j \to f$ strongly in $L^2_{t,x} H^1_v$ and $\trp f_j \to \trp f$ weakly in $L^2_{t,x} H^{-1}_v$, then $f_j \to f$ strongly in $L^2_{loc}(\tilde \gamma,\omega)$.
\end{proposition}

\begin{proof}
We prove the inequality for $f$ smooth. From Lemma \ref{l:smooth-approx}, every $f \in \Hk(\HH_1(\z))$ can be approximated with smooth ones. By density, this defines the trace operator in the space $\Hk$.

Let $\varphi_+$ and $\varphi_-$ be nonnegative functions defined initially on $\gamma$ as
\begin{align}
\varphi_+(t,x,v) &= \min((v \cdot n)_+,1), \\
\varphi_-(t,x,v) &= \min((v \cdot n)_-,1).
\end{align}

They are Lipschitz functions provided that the domain $\Omega$ has a $C^{1,1}$ boundary. Our choice of the functions $\varphi_\pm$ is so that $(\varphi_+ - \varphi_-) \cdot (n \cdot v)$ is nonnegative on $\gamma$ and $\approx \omega$.

Let us extend $\varphi_\pm$ to all of $\HH_1(\z)$ as Lipschitz functions. We want to estimate the integral of $f \varphi_\pm (v\cdot n)$ on $\gamma$. 

We start with $\gamma_+$. Let $\eta$ be a smooth function that equals one on $\tilde \gamma$ and equals zero on $\partial Q_1$. We integrate $\trp(\varphi_+ \eta f^2)$ in $\HH_1(\z)$ to obtain
\begin{align*}
\int_\gamma \varphi_+ \eta f^2 (n \cdot v) \dd \gamma & =
\int_{\HH_1(\z)} \trp(\eta \varphi_+ f^2) \dvxt \\
&= \int_{\HH_1(\z)} \trp(\eta \varphi_+) f^2 + 2 \eta \varphi_+ f \trp f  \dvxt \\
&\leq \|\trp(\eta \varphi_+)\|_{L^\infty} \|f\|_{L^2}^2 + \|2 \eta \varphi_+ f\|_{L^2_{t,x} H^1_v} \|\trp f\|_{L^2_{t,x} H^{-1}_v} \\
&\leq C \|f\|_{\Hk(\HH_1(\z_0))}^2.
\end{align*}

Because of our choices of $\varphi_+$ and $\eta$, we have that $\varphi_+ \eta (n \cdot v) \geq 0$ and it equals $\omega$ on $\gamma_+ \cap \tilde \gamma$. The computation above gives us our desired bound on the outgoing part of the boundary. Similarly, we integrate $-\trp(\varphi_- \eta f^2)$ to obtain a bound on $\gamma_-$. 

Observing that $\trp f$ is multiplied with $2 \eta \varphi_+ f$ in the integral above, we deduce the last assertion about the strong convergence of the trace on $L^2(\gamma_0,\omega)$, by a standard weak-strong pairing argument.
\end{proof}

\begin{remark}
In \cite[Question 1.8]{albritton2019variational}, the authors conjecture a stronger trace inequality than the one provided by Proposition \ref{p:trace}. Essentially, they claim that the inequality may still hold with the simpler weight $\omega = |v \cdot n|$. We have not been able to either prove or disprove their conjecture. Consequently, we do not know if the weight $\omega$ in Proposition \ref{p:trace} is sharp. 
\end{remark}

%\begin{remark}
%A small refinement of the computation in the proof of Proposition \ref{p:trace} shows that the trace operator on $\gamma_+$ can be defined up to the final time in $\HH_1(\z)$ (that is on $\gamma \cap Q_{1-\eps}(\z)$). Likewise, the trace on $\gamma_-$ can be defined from the initial time of $\HH_1(\z)$.
%\end{remark}

%\luis{The next lemma about continuity with respect to one spatial variable has to be reformulated}
%
%\begin{proposition} \label{p:continuous-trace}
%Consider the domain $\Omega = \{x_1 < 0\}$, with the flat boundary.
%
%Let $f \in \Hk(\HH_1(\z))$. Then $f$ is almost everywhere equal to a continuous function from $x_1 \in (-1,0]$ with values %in $L^2((-1+\eps,0-\eps] \times B_{1-\eps}' \times B_{1-\eps}, \omega)$, for any $\eps>0$.%
%
%In other words, the space $\Hk(\HH_1(\z))$ is continuously embedded in $C_{x_1} L^2_{t,x',v,\text{loc}}(\omega)$.
%\end{proposition}

One possible way to make sense of the inflow boundary condition $f = g$ on $\gamma_-$ would be the following. Assume that $\trp f \in L^2_{t,x} H^{-1}_v$. Then, for all $\varphi \in C^1_c(Q_1(\z))$ so that $\varphi = 0$ on $\gamma_+$ we have
\begin{equation} \label{e:weak-boundary}
\int (\partial_t + v \cdot \nabla_x) f \, \varphi + f \, (\partial_t + v \cdot \nabla_x) \varphi \dvxt= \int_{\gamma_-} g \varphi (v \cdot n) \dd \gamma.
\end{equation}

\begin{lemma} \label{l:weak-trace-zero}
Assume $f \in \Hk(\HH_1(\z))$ and \eqref{e:weak-boundary} holds for any $\varphi \in C^1_c(Q_1(\z))$ so that $\varphi = 0$ on $\gamma_+$. Let $f_{|\gamma}$ be the function in $L^2(\gamma,\omega)$ described in Proposition \ref{p:trace}. Then $f_{|\gamma} = g$ on $\gamma_-$.
\end{lemma}

\begin{proof}
When $f$ is a smooth function, the formula \eqref{e:weak-boundary} is a standard integration by parts with $g = f_{|\gamma_-}$.

Since $\varphi=0$ on $\gamma_+$, then $|\varphi (v \cdot n)| \leq \min(|v\cdot n|, |v\cdot n|^2)$. Therefore, the boundary integral
\[ g \mapsto \int_{\gamma_-} g \varphi (v \cdot n) \dd \gamma, \]
is a bounded linear operator on the space $L^2(\bar \gamma_-,\omega)$, where $\bar \gamma_- = \gamma_- \cap \supp \varphi$.

Because of Proposition \ref{p:trace}, we get that 
\[ f \mapsto \int_{\bar \gamma_-} f_{|\gamma_-} \varphi (v \cdot n) \dd \gamma, \]
is a continuous linear functional on $\Hk(\HH_1(\z))$. We deduce that every term in \eqref{e:weak-boundary} is continuous on $\Hk(\HH_1(\z))$ as a function of $f$. Since the identity holds when $f$ is smooth with $g = f_{|\gamma_-}$, it must hold for every $f \in \Hk(\HH_1(\z))$ with $g = f_{|\gamma_-}$ by density.
\end{proof}

At some point in this article, we will want to test a function $f$ against a test function $\varphi \in C^1_c(\z)$ that does not vanish on $\gamma_+$. Assuming that both functions are nonnegative, we have an inequality resulting from the missing boundary term on $\gamma_-$. The following lemma applies in particular when $f \in \Hk(\HH_1(\z))$, but also in more general situations.

\begin{lemma} \label{l:boundary-loss}
Let $f \in L^2(\HH_1(\z))$ be a nonnegative function so that
\[ \trp f = -\mu + \zeta,\]
where $\mu$ is a signed measure with finite total variation in $\HH_1(\z)$, and $\zeta \in L^2_{t,x} H^{-1}_v$. Assume that for any nonnegative test function $\varphi \in C^1_c(Q_1(\z))$, so that $\varphi = 0$ on $\gamma_+$, we have
\[ \int_{\HH_1(\z)} \trp f \, \varphi + f \, \trp \varphi \dvxt \geq 0. \]

If $\gamma\in C^1_c(Q_1(\z))$ is any nonnegative test function (that may not vanish on $\gamma_+$), then we also have the same inequality.
\end{lemma}

In Lemma \ref{l:boundary-loss}, when we write $\trp f$, we mean the directional derivative of $f$ in the sense of distributions in the open set $\HH_1(\z)$. Note that the main assumption of the lemma would mean that $f=0$ on $\gamma_-$ if $f$ was a smooth function.

\begin{proof}
Let $\tilde \eta_1:[0,\infty) \to \R$ be a smooth function equal to zero in a neighborhood of zero, and equal to one on $[1,\infty)$. Let $\eta_\eps(z) = \tilde\eta_1(\dist(z,\gamma_+)/\eps)$, where $\dist$ denotes the usual Euclidean distance in $\R^{1+2d}$.\footnote{It is usually a bad idea to use the usual Euclidean distance in the context of kinetic equations. For the purpose of building the test function in this proof, either choice of distance works fine. We thought that the Euclidean distance gives us an easier geometry to understand intuitively in this case.} Consider $\varphi_\eps = \eta_\eps \varphi$. Since $\varphi_\eps = 0$ on $\gamma_+$, we have
\begin{equation} \label{pp:b1}
\int (\partial_t + v \cdot \nabla_x) f \, \varphi_\eps + f \, (\partial_t + v \cdot \nabla_x) \varphi_\eps \dvxt \geq 0.
\end{equation}
Clearly, $\varphi_\eps \to \varphi$ in $L^2(Q_1)$. We claim that also $\varphi_\eps \to \varphi$ weakly in $L^2_{t,x} H^1_v$. Indeed, we have $\nabla_v \varphi_\eps = (\nabla_v \varphi)\eta_\eps + \varphi \nabla_v \eta_\eps$. The convergence of the first term in $L^2$ is trivial because $\nabla_v \varphi$ is a fixed bounded function and $\eta_\eps \to 1$ in $L^2$. The second term is slightly more delicate. The derivative $\nabla_v \eta_\eps$ is $\lesssim \eps^{-1}$. It is also supported in a neighborhood of diameter $\eps$ around $\gamma_0$. Thus
\[ \| \varphi \nabla_v \eta_\eps \|_{L^2}^2 \lesssim \eps^{-2} |\{ z : \dist(z,\gamma_0) < \eps \}| \approx 1.\]
Thus, $\varphi \nabla_v \eta_\eps$ is bounded in $L^2$. We cannot say it converges to zero in $L^2$, but since its support is contained in $\{ z : \dist(z,\gamma_0) < \eps \}$ that shrinks to measure zero, it converges to zero weakly in $L^2$.

Recall that $\trp f = \zeta - \mu$. Since $\varphi_\eps \to \varphi$ weakly in $L^2_{t,x} H^1_v$, then
\[ \int \zeta \, \varphi_\eps \dvxt \to \int \zeta \, \varphi  \dvxt . \]

Since $\varphi \geq 0$, we have that $\varphi_\eps$ converges monotonically to $\varphi$ in $\HH_1(\z)$. We assume that $\mu$ has finite total variation. By the dominated convergence theorem,
\[ \int \mu \, \varphi_\eps \dvxt \to \int \mu \, \varphi  \dvxt . \]
Combining the last two limits, we get
\[ \int \trp f \, \varphi_\eps \dvxt \to \int \trp f \, \varphi  \dvxt . \]

Let us move to the next term in \eqref{pp:b1}. We have $\trp \varphi_\eps = \eta_\eps \trp \varphi + \varphi \trp \eta_\eps$. The first term clearly converges in $L^2$. For the second term, we see that $\partial_t \eta_\eps = 0$. The key observation is that by construction $v \cdot \nabla_x \eta_\eps \leq 0$ except at most in a neighborhood of $\gamma_0$ of diameter $\eps$, where we have $v \cdot \nabla_x \eta_\eps \lesssim \eps^{-1}$. Therefore
\begin{align*}
\int f \, \varphi \, (\partial_t + v \cdot \nabla_x) \eta_\eps \dvxt &\lesssim \eps^{-1} \int_{Q_1 \cap \{\dist(z,\gamma_0) < \eps\}} f \dvxt \\
&\leq \eps^{-1} \left(\int_{Q_1 \cap \{\dist(z,\gamma_0) < \eps\}} f^2 \dvxt\right)^{1/2} |\{Q_1 \cap \{\dist(z,\gamma_0) < \eps\} \}|^{1/2} \\
&= c \left(\int_{Q_1 \cap \{\dist(z,\gamma_0) < \eps\}} f^2 \dvxt\right)^{1/2} \to 0 \text{ as } \eps \to 0.
\end{align*}

Using product rule in \eqref{pp:b1}, we have
\[ \int (\partial_t + v \cdot \nabla_x) f \, \varphi_\eps + f \, \eta_\eps \, (\partial_t + v \cdot \nabla_x) \varphi \dvxt \geq -\int f \, \varphi \, (\partial_t + v \cdot \nabla_x) \eta_\eps \dvxt.
\]
Taking lim-inf as $\eps \to 0$, we finish the proof.

\end{proof}

Here, we also recall some well known hypoelliptic estimates for the kinetic Sobolev spaces. A proof of the following two propositions can be found, for example, in \cite{albritton2019variational}.

\begin{proposition} \label{p:interior-hypoelliptic}
Let $D \subset \R^{1+2d}$ be an open set and $D_1$ be compactly contained in $D$. There exists an $s>0$, depending on dimension only, so that we have
\[ \|f\|_{W^{s,2}(D_1)} \leq C \|f\|_{\Hk(D)},\]
for some constant $C$ depending on the domains $D$ and $D_1$.
\end{proposition}

As a corollary, we state the compactness of the embedding $\Hk \hookrightarrow L^2_{loc}$.

\begin{proposition} \label{p:compact-embedding}
Let $D \subset \R^{1+2d}$ be an open set and $D_1$ be compactly contained in $D$. The embedding $\Hk(D) \hookrightarrow L^2(D_1)$ is compact.
\end{proposition}

\section{The notion of solution}

We start with the following rather weak notion of solution.

\begin{definition} \label{d:weak-sol}
Let $D$ be an open subset of $\R \times \Omega \times \R^d$. We say that a function $f \in L^2(D)$ is a weak solution of \eqref{e:kFP} in $D$ if $\nabla_v f \in L^2(D)$ and for any test function $\varphi \in C^1_c(\R^{1+2d})$ so that $\supp \varphi \cap \partial D \subset \gamma_- \cup \gamma_0$, we have
\[ \begin{aligned}
\int_D - f (\partial_t + v \cdot \nabla_x) \varphi + a_{ij} \partial_{v_j} f \partial_{v_i} \varphi + b_i \partial_{v_i} f \varphi - G \varphi \dvxt \\
= -\int_{\gamma_-} g \varphi (v \cdot n) \dd \gamma.
\end{aligned}
\]
\end{definition}

Definition \ref{d:weak-sol} applies in particular for test functions $\varphi$ that are compactly supported inside $D$. In that case, it literally says that \eqref{e:kFP} holds in the sense of distributions. Since $\nabla_v f \in L^2$, we observe that every term other than transport in \eqref{e:kFP} belongs to $L^2_{t,x} H^{-1}_v$. Thus, Definition \ref{d:weak-sol} immediately implies that $f \in \Hk$, even if it is not explicitly given as a requirement.

The fact that we allow the support of the test function $\varphi$ in Definition \ref{d:weak-sol} to intersect the physical boundary $\partial \Omega$ is used to encode the boundary condition. Indeed, if $f$ was a classical solution of \eqref{e:kFP}, we integrate by parts and obtain
\begin{align*}
\int - f (\partial_t + v \cdot \nabla_x) \varphi &+ a_{ij} \partial_{v_j} f \partial_{v_i} \varphi + b_i \partial_{v_i} f \varphi- G \varphi  \dvxt = \\
&= \int \left\{ (\partial_t + v \cdot \nabla_x) f - \partial_{v_i} (a_{ij}(t,x,v) \partial_{v_j} f) + b_i \partial_{v_i} f \varphi - G  \right\} \varphi \dvxt \\
& \phantom{=} - \int_{\{ x \in \partial \Omega\}} f \varphi (n \cdot v) \dd v \dd S(x) \dd t
\end{align*}
For the boundary integral, we have $\varphi=0$ on $\gamma_+$ and $f = g$ on $\gamma_-$.

We remark that Definition \ref{d:weak-sol} is equivalent to stating that $f \in \Hk(D)$, \eqref{e:kFP} holds in the sense of distributions, and \eqref{e:weak-boundary} holds.

\subsection{Extending subsolutions as zero}

Another way to interpret the boundary condition $f = 0$ on $\gamma_-$ for weak solutions is by extending $f$ as zero and observing that it will satisfy the equation across $\gamma_-$. The next lemma shows how these definitions are equivalent.

\begin{lemma} \label{l:extending-as-zero}
Let $D$ be an open subset of $\R \times \Omega \times \R^d$ and $f: D \to \R$. Let $D \subset \tilde D \subset \R^{1+2d}$. Assume that $\tilde D \cap \partial D \subset \gamma_-$. Let us extend the function $f$ to all of $\tilde D$ as zero. We also extend the vector field $b$ and the coefficients $a_{ij}$ in any arbitrary way (maintaining the equation parameters), and the source term $G$ as zero.  If $f$ is a weak solution of \eqref{e:kFP} in the sense of Definition \ref{d:weak-sol} with $f=0$ on $\gamma_-$, then the extended function satisfies the equation \eqref{e:kFP} in the sense of distributions in all of $\tilde D$.

Moreover, when all of $\gamma_- \cap D$ is contained in $\tilde D$, the two statements are equivalent.
\end{lemma}

\begin{proof}
The proof is a straighforward verification that both definitions coincide.

Let $\varphi \in C^1_c(D)$ be any test function. Since $f$ and $G$ are extended as zero, we verify that
\[ \begin{aligned}
\int_{D} - f (\partial_t + v \cdot \nabla_x) \varphi + a_{ij} \partial_{v_j} f \partial_{v_i} \varphi + b_i \partial_{v_i} f \varphi - G \varphi  \dvxt = \\
\int_{\tilde D} - \tilde f (\partial_t + v \cdot \nabla_x) \varphi + a_{ij} \partial_{v_j} \tilde f \partial_{v_i} \varphi + b_i \partial_{v_i} \tilde f \varphi - G \varphi  \dvxt. 
\end{aligned}
\]
\end{proof}

Similarly as in Definition \ref{d:weak-sol}, we give a notion of \emph{weak subsolution}. Since we only plan to study nonnegative subsolutions with zero influx boundary conditions, we restrict the definition to that class of functions. This is convenient because of Lemma \ref{l:boundary-loss}.

\begin{definition} \label{d:weak-sub-sol}
Let $D$ be an open subset of $\R \times \Omega \times \R^d$. We say that a \textbf{nonnegative} function $f \in L^2(D)$ is a weak subsolution of \eqref{e:kFP} in $D$, with $f=0$ on $\gamma_-$, if $\nabla_v f \in L^2(D)$, 
and for any \textbf{nonnegative} test function $\varphi \in C^1_c(\R^{1+2d})$ so that $\supp \varphi \cap \partial D \subset \gamma$, we have
\[ \int - f (\partial_t + v \cdot \nabla_x) \varphi + a_{ij} \partial_{v_j} f \partial_{v_i} \varphi + b_i (\partial_{v_i} f) \varphi  - G \varphi  \dvxt \leq 0.\]
\end{definition}

Note that the test function $\varphi$ in Definition \ref{d:weak-sub-sol} is not required to vanish on $\gamma_+$. A nonnegative solution $f$, with $g=0$, is also a subsolution in the sense of Definition \ref{d:weak-sub-sol} thanks to Lemma \ref{l:boundary-loss}.

\begin{lemma} \label{l:extending-sub-as-zero}
Let $D$ be an open subset of $\R \times \Omega \times \R^d$ and $f: D \to \R$. Let $D \subset \tilde D \subset \R^{1+2d}$. Assume that $\tilde D \cap \partial D \subset \gamma_-$. Let us extend the function $f$ to all of $\tilde D$ as zero (and call it $\tilde f$). We also extend the vector field $b$ and the coefficients in any arbitrary way (maintaining the equation parameters), and the source term $G$ as zero. If $f$ is a nonnegative weak subsolution of \eqref{e:kFP} in the sense of Definition \ref{d:weak-sub-sol}, with $f = 0$ on $\gamma_-$, then the extended function satisfies the inequality in the sense of distributions
\[ (\partial_t + v\cdot \nabla_x) f - \partial_{v_i} (a_{ij}(t,x,v) \partial_{v_j} f) + b_i \partial_{v_i} f \leq G  \text{ in all } \tilde D.\]

Moreover, when all of $\gamma_- \cap D$ is contained in $\tilde D$, the two statements are equivalent.
\end{lemma}

The proof of Lemma \ref{l:extending-sub-as-zero} is a straight forward verification like that of Lemma \ref{l:extending-as-zero}.

% \luis{Arguably, there is some justification needed as to why the last statement involves $\gamma_-$ only instead of all $\gamma$}

%If we have a solution or a subsolution in some domain $D \subset \R^{1+2d}$, we may extend it beyond the incoming part of the boundary $\gamma_-$. However, there is no clear systematic choice of the extended domain $\tilde D$ in general. When $\Omega$ is convex, there is a canonical choice. Let us say $D = [T_1,T_2] \times \Omega \times \mathcal V$, for any open set $\mathcal V \subset \R^d$. We can choose $\tilde D$ in the following way
%\[ \tilde D := \{ (t,x,v) \in [T_1,T_2] \times \R^d \times \mathcal V: (t+s,x+sv,v) \in D \text{ for some } s \geq 0\}.\]
%One can easily verify that $\tilde D \cap \partial D = \gamma_-$ when $\Omega$ is convex.

%The extended domain $\tilde D$ consists of all points that would eventually flow into $D$. We call it the \textbf{predecesor of $D$} \footnote{or should we call it precursor?}. It is important to note that the incoming part of the boundary $\gamma_-$ is completely in the interior of $\tilde D$. Because of this, Definitions \ref{d:weak-sol} and \ref{d:weak-sub-sol} are equivalent as the statement that the zero-extension to the predecesor domain $\tilde D$ satisfies the equation (or inequality) in the sense of distributions.

It is useful to observe that when $\gamma_-$ is empty, then the boundary condition is irrelevant and the definitions given above reduce to the usual notion of solution in the sense of distributions.

\begin{lemma} \label{l:irrelevant-boundary}
Assume $\gamma_- \cap Q_1(\z) = \emptyset$. Let $f$ be a solution of \eqref{e:kFP} in $\HH_1(\z)$, in the sense of distributions. Then $f$ is also a solution in the sense of Definition \ref{d:weak-sol}.

Also, if $f$ is a nonnegative function and a subsolution of \eqref{e:kFP} in the sense of distributions. Then $f$ is also a nonnegative subsolution in the sense of Definition \ref{d:weak-sub-sol}.
\end{lemma}

\begin{proof}
In the case of solutions, the lemma is straightforward since the test function $\varphi$ in Definition \ref{d:weak-sol} must vanish on $\gamma_+$. In this case, it means that $\varphi$ must vanish in all $\partial D$. The notion of weak solution in $D$ consists in testing with the same family of test functions.

In the case of subsolutions, Definition \ref{d:weak-sub-sol} involves test functions that do not necessarily vanish on $\gamma_+$. Definition \ref{d:weak-sub-sol} trivially implies that $f$ is a solution inside $D$ in the sense of distributions because it involves a larger family of test functions. The opposite implication follows from Lemma \ref{l:boundary-loss}.
\end{proof}

\subsection{Chain rule for weak solutions}

It will be important to compute the equation satisfied by functions of the form $\psi(f)$, for certain choices of $\psi$. The following lemma tells us that we can apply the usual chain rule to solutions in the sense of Definition \ref{d:weak-sol}.

\begin{lemma} \label{l:weak-chain-rule}
Let $f$ be a weak solution of \eqref{e:kFP} with boundary data \eqref{e:influx} in the sense of Definition \ref{d:weak-sol}. Assume $G$ is bounded. Let $\psi \in C^{2}(\R)$ and $\psi''$ bounded. Let $\varphi \in C^1_c(\R^{1+2d})$ be a test function so that $\supp \varphi \cap \partial D \subset \overline{\gamma_-}$. Then
\begin{equation} \label{e:renormalization}
 \begin{aligned}
&\int  -\psi(f) \trp \varphi + \psi''(f) (a_{ij} \partial_{v_i} f \partial_{v_j} f) \varphi \dvxt
\\
&+ \int a_{ij} \partial_{v_j} \psi(f) \partial_{v_i} \varphi + b_i \partial_{v_i} \psi(f) \varphi - G \psi'(f) \varphi \dvxt \\
&= -\int_{\gamma_-} \psi(g) \varphi (v \cdot n) \dd \gamma .
\end{aligned} 
 \end{equation} 
\end{lemma}

% {\color{red} Can we do the same lemma with $\psi \in C^{1,1}$?}

\begin{proof}
For any open cover of $D$, we can decompose the test function $\varphi$ into a sum of test functions supported inside each open set of the cover. Using the technique of section \ref{s:flattening}, we reduce the problem to verify the assertion locally around a flat boundary. We can then assume without loss of generality that $\Omega = \{ x_1 < 0 \}$ and $D = \HH_1(\bar z)$. Moreover, the function $\varphi$ is supported in $Q_1(\bar z)$ and $\varphi = 0$ on $\gamma_+$.

%For technical reasons, let us assume first that $\psi$ is a bounded function.

The function $f$ satisfies the equation \eqref{e:kFP} in the sense of distributions. Let us define the mollifications $f_\eps = \eta_\eps \ast f$ like in the proof of Lemma \ref{l:smooth-approx}. Recall that by choosing a function $\eta_\eps$ supported in $\{x_1 > 0, v_1 < 0, t>0\}$ we ensure that the values of $f_\eps(t,x,v)$, for $x_1 \leq 0$, depend only on values of $f$ in $\{ x_1 < 0\}$.

The equation \eqref{e:kFP} holds in the sense of distributions inside $\HH_1$. Let us convolve the whole expression with $\eta_\eps$. Recalling that convolution commutes with $\partial_{v_i}$ and $\trp$, we obtain
\[ \trp f_\eps - \partial_{v_i} \left( \eta_\eps \ast (a_{ij} \partial_{v_j} f) \right) = \eta_\eps \ast G.\]
We multiply this expression times $\psi'(f_\eps) \varphi$ and integrate by parts. We get
\begin{align*}
0 = \int_{\HH_1} & - \psi(f_\eps) \trp \varphi + \psi''(f_\eps) (\eta_{\eps}\ast(a_{ij} \partial_{v_j} f)) (\partial_{v_i} f_\eps) \varphi + \psi'(f_\eps) (\eta_{\eps}\ast(a_{ij} \partial_{v_j} f)) \partial_{v_i} \varphi \\
& + \eta_\eps \ast (b_i \partial_{v_i} f) \psi'(f_\eps) \varphi - (\eta_\eps \ast G) \psi'(f_\eps) \varphi \dvxt \\
+ \int_\gamma &\psi(f_\eps) \varphi (v \cdot n) \dd \gamma.
\end{align*}

Now, we want to take $\eps \to 0$ and pass to the limit every term in the integral expression. We observe the following.
\begin{itemize}
    \item $f_\eps \to f$ in $L^2_{loc}$. Consequently, also $\psi(f_\eps) \to \psi(f)$ in $L^1_{loc}$ and $\psi'(f_\eps) \to \psi'(f)$ in $L^2_{loc}$.
    \item $\psi''(f_\eps) \to \psi''(f)$ almost everywhere. Moreover, $\psi''(f_\eps)$ is uniformly bounded since we assumed that $\psi''$ is bounded.
    \item $\partial_{v_i} f_\eps \to \partial_{v_i} f$ in $L^2_{loc}$.
    \item $\eta_\eps \ast (a_{ij} \partial_{v_j} f) \to a_{ij} \partial_{v_j} f$ in $L^2_{loc}$.
    \item $\eta_\eps \ast (b_i \partial_{v_i} f) \to b_i \partial_{v_i} f$  in $L^2_{loc}$.
    \item $\eta_\eps \ast G \to G$ almost everywhere, and it is uniformly bounded.
\end{itemize}
Using these observations, we are able to pass to the limit every term in the first integral using the dominated convergence theorem. Thus
\begin{align*}
 \int &-\psi(f) \trp \varphi + \psi''(f) a_{ij} \partial_{v_i} f \partial_{v_j} f \varphi + a_{ij} \partial_{v_j} \psi(f) \partial_{v_i} \varphi \\
 &+b_i \partial_{v_i} \psi(f) - G \psi'(f) \varphi \dvxt= \lim_{\eps \to 0} \int_\gamma \psi(f_\eps) \varphi (v \cdot n) \dd \gamma.
\end{align*}
We are left to analyze the boundary term. Since $f \in \Hk$, we observe by duality that $\trp f_\eps \to \trp f$ at least weakly in $L^2_{t,x} H^{-1}_{v, \text{loc}}$. From Proposition \ref{p:trace}, we have that $f_\eps \to g$ strongly in $L^2_{loc}(\gamma,\omega)$ (since $f_{|\gamma_-} =g$). We assume that $\psi''$ is bounded, so $\psi(f_\eps) \to \psi(g)$ in $L^1(\omega)$.

Since $\varphi=0$ on $\gamma_+$, we have $|\varphi (v \cdot n)| \lesssim \omega$ on $\gamma$. We deduce the convergence of the boundary term.

%In particular (taking a further subsequence if necessary), $f_\eps \to g$ almost everywhere on $\gamma_-$. Since we assume %that $\psi$ is a bounded function, we have that $\psi(f_\eps)$ is uniformly bounded. Since $\varphi = 0$ on $\gamma_+$, we %apply the dominated convergence theorem near $\gamma_0$ and get
\[ \lim_{\eps \to 0} \int_\gamma \psi(f_\eps) \varphi (v \cdot n) \dd \gamma = \int_\gamma \psi(g) \varphi (v \cdot n) \dd \gamma. \]

%This finishes the proof of the lemma in the case that $\psi$ is bounded. If $\psi$ is unbounded, we consider a sequence of bounded $C^2$ functions $\psi_n$ so that $\psi_n''$ is uniformly bounded and $\psi_n = \psi$ in $[-n,n]$. We know that \eqref{e:renormalization} holds for every $\psi_n$. It is easy to verify that every term converges as $n \to \infty$.

%\luis{Do we really need to start with $\psi$ bounded any more?}
\end{proof}

We extend Lemma \ref{l:weak-chain-rule} as an inequality for subsolutions.

\begin{lemma} \label{l:sub-chain-rule}
Let $f$ be a nonnegative weak subsolution of \eqref{e:kFP} in $\HH_1(\z)$, with $f=0$ on $\gamma_-$, in the sense of Definition \ref{d:weak-sub-sol}. Assume $G$ is bounded. Let $\psi \in C^{2}(\R)$, with $\psi(0) = 0$, $\psi(f) \geq 0$ and $\psi''$ bounded. Let $\varphi \in C^1_c(Q_1)$ be any test function. Then
\begin{equation} 
 \begin{aligned}
\int -\psi(f) \trp \varphi + \psi''(f) (a_{ij} \partial_{v_i} f \partial_{v_j} f) \varphi + a_{ij} \partial_{v_j} \psi(f) \partial_{v_i} \varphi & \\
+ b_i \partial_{v_i} \psi(f) \varphi - G \psi'(f) \varphi &\leq 0 .
\end{aligned} 
\end{equation} 
\end{lemma}

\begin{proof}
The proof follows the same lines as the proof of Lemma \ref{l:weak-chain-rule}. This case is simpler because there is no boundary term and the test function $\varphi$ is not required to vanish on any part of $\gamma$.
%\luis{Wait! What happens with the boundary term? Here is an alternative.}
%Let us start with a test function $\varphi$ whose support does not intersect with $\overline{\gamma_+}$. Use Lemma \ref{l:extending-sub-as-zero} to extend the function $f$ to all $Q_1(\z) \setminus \overline{\%gamma_+}$. Now we mollify and proceed like in the proof of Lemma \ref{l:weak-chain-rule}. Note that because there is a positive gap between the support of $\varphi$ and $\gamma_+$, the mollification does not %touch the values of $f$ on $\overline \gamma_+$ at a sufficiently small scale.%
%
%Once we have the inequality for those functions $\varphi$, it is straight forward to extend it to anyone that vanishes on $\overline \gamma_+$.
\end{proof}

When there is no incoming boundary in $\HH_1(\z)$, we do not need to assume $\psi(0)=0$.

\begin{lemma} \label{l:sub-chain-rule-no-incoming}
Assume $\gamma_- \cap Q_1(\z) = \emptyset$. Let $f$ be a weak subsolution of \eqref{e:kFP} in $\HH_1(\z)$, in the sense of distributions. Assume $G$ is bounded. Let $\psi \in C^{2}(\R)$, with $\psi''$ bounded and $\psi(f) \geq 0$. Let $\varphi \in C^1_c(Q_1)$ be a test function compactly supported in $Q_1$. Then
\begin{equation} 
 \begin{aligned}
\int_{\HH_1(\z)} -\psi(f) \trp \varphi + \psi''(f) (a_{ij} \partial_{v_i} f \partial_{v_j} f) \varphi + a_{ij} \partial_{v_j} \psi(f) \partial_{v_i} \varphi & \\
+ b_i \partial_{v_i} \psi(f) \varphi - G \psi'(f) \varphi \dvxt &\leq 0 .
\end{aligned} 
\end{equation} 
\end{lemma}

\begin{proof}
If the test function $\varphi$ is supported inside $\HH_1(\z)$, the result follows the steps of the proof of Lemma \ref{l:sub-chain-rule}. Once we established the proof for functions supported inside $\HH_1(\z)$, by density it implies that the inequality holds for $C^1$ functions that vanish on $\gamma$. Then we extend it to any test function supported in $Q_1$ using that $\gamma \cap Q_1 = \gamma_+ \cap Q_1$ and Lemma \ref{l:boundary-loss}.
\end{proof}

\begin{lemma} \label{l:convex-sub}
Let $f$ be a weak solution of \eqref{e:kFP} in $\HH_1(\z)$, with boundary condition \eqref{e:influx}, in the sense of Definition \ref{d:weak-sol}. For any $m \in \R$ so that $g-m \leq 0$ for every $z \in \gamma_-$, the function $(f-m)_+$ is a nonnegative weak subsolution of \eqref{e:kFP} with $G \one_{f > m}$ instead of $G$.
\end{lemma}

\begin{proof}
Lemma \ref{l:weak-chain-rule} applies only for $\psi \in C^2$. In this case, the function $\psi(f) = (f-m)_+$ is not $C^2$. The trick of this proof is to use its convexity to drop the term that depends on $\psi''$, retaining an inequality.

For $\eps>0$, let $\psi_\eps(f)$ be the following approximation of $(f-m)_+$.
\[ \psi_\eps(f) := \begin{cases}
(f-m-\eps) &\text{if } f>m+2\eps, \\
0 &\text{if } f \leq m,\\
\text{something smooth and convex} &\text{otherwise}.
\end{cases}
\]
We further construct $\psi_\eps$ so that $0\leq \psi_\eps'(f) \leq 1$ for all values of $f$. Note that $\lim_{\eps \to 0} \psi_\eps(f) = (f-m)_+$. 

Applying Lemma \ref{l:weak-chain-rule}, for any nonnegative test function $\varphi$ that vanishes on $\gamma_+$ we get
\begin{equation} \label{pf:abs1}
\begin{aligned}
\int & -\psi_\eps(f) \trp \varphi + \psi_\eps''(f) (a_{ij} \partial_{v_i} f \partial_{v_j} f) \varphi + a_{ij} \partial_{v_j} \psi_\eps(f) \partial_{v_i} \varphi - G \psi_\eps'(f) \varphi \dvxt= 0 .
\end{aligned} 
 \end{equation}
The term $\psi_\eps''(f) (a_{ij} \partial_{v_i} f \partial_{v_j} f)$ is nonnegative because $\psi_\eps$ is convex. Therefore, we can drop it and retain an inequality.
\begin{equation} \label{pf:abs2} \begin{aligned}
\int & -\psi_\eps(f) \trp \varphi + a_{ij} \partial_{v_j} \psi_\eps(f) \partial_{v_i} \varphi - G \psi_\eps'(f) \varphi \dvxt \leq 0 .
\end{aligned} 
\end{equation}
The identity \eqref{pf:abs1} means that the following equation holds in the sense of distributions in $\HH_1(\z)$.
\[ \trp \psi_\eps(f) + \psi_\eps''(f) (a_{ij} \partial_{v_i} f \partial_{v_j} f) - \partial_{v_i} \left\{ a_{ij} \partial_{v_j} \psi_\eps(f) \right\} = G \psi_\eps'(f).\]
In particular, we can apply Lemma \ref{l:boundary-loss} and deduce that \eqref{pf:abs2} holds also for any $\varphi \geq 0$ that may not necessarily vanish on $\gamma_+$.

We proceed to take the limit as $\eps \to 0$ and finish the proof.
\end{proof}

When there is no incoming boundary in the domain, we do not need to assume anything about $\psi(g)$. The following corollary states explicitly that particular case of Lemma \ref{l:convex-sub}.

\begin{corollary} \label{c:convex-sub-no-influx}
Let $f$ be a weak solution of \eqref{e:kFP} in $\HH_1(\z)$, in the sense of Definition \ref{d:weak-sol}. Assume that $\gamma_- \cap Q_1(\z) = \emptyset$. For any $m \in \R$, the function $(f-m)_+$ is a nonnegative weak subsolution of \eqref{e:kFP} with $G \one_{f > m}$ instead of $G$.
\end{corollary}

\section{Estimates without boundary.}
\label{s:interior}

We review the methods that lead to interior H\"older estimates for kinetic Fokker-Planck equations like \eqref{e:kFP}. The results in this section are proved in \cite{polidoro2004,zhang2008,wangzhang2009,zhang2011,golse2019harnack,guerand2021quantitative}.

The first step in the De Giorgi method is an upper bound of the $L^\infty$ norm of a subsolution in terms of its $L^2$ norm. It is the interior version of our Theorem \ref{t:Linfty}. It was first proved in \cite[Theorem 1.2]{polidoro2004} using Moser's iteration, for a general family of ultraparabolic equations, but without a right-hand side. New proofs were given in \cite[Theorem 12]{golse2019harnack} and \cite[Proposition 12]{guerand2021quantitative} for kinetic equations with a right-hand side in $L^\infty$. Also, a version for more general Kolmogorov equations is given in \cite[Theorem 3.1]{anceschi2021note}. We state the result without proof in terms of our notation.

\begin{proposition} \label{p:interior-Linfty}
Let $f : Q_1 \to \R$ be a weak subsolution of the equation \eqref{e:kFP}. Then
\[ \esssup_{Q_{1/2}} f_+ \leq C \left( \|f_+\|_{L^2(Q_1)} + \|G\|_{L^\infty(Q_1)} \right),\]
for a constant $C$ depending on the ellipticity parameters of the coefficients, $\|b\|_{L^\infty}$, and dimension only.
\end{proposition}

Beyond Proposition \ref{p:interior-Linfty}, the standard proof of the interior H\"older continuity result follows after iterating a gain-of-oscillation lemma. 

Let us define
\[ Q^- := (-3/4,-1/2] \times B_{1/2^3} \times B_{1/2}.\]
The important property of $Q^-$ is that it is compactly contained in $Q_1$ and all its points are earlier in time than $Q_{1/2}$. The interior H\"older continuity of the solution is deduced using the following \emph{growth} lemma.

The following lemma is essentially the same as \cite[Lemma 18]{golse2019harnack} or \cite[Lemma 16]{guerand2021quantitative}. 

\begin{lemma} \label{l:growth-lemma}
For every $\mu>0$, there is an $\eps_0>0$ so that the following statement is true. Let $f : Q_1 \to \R$ be a subsolution of the equation \eqref{e:kFP}. Assume the following
\begin{itemize}
\item $f \leq 1$ in $Q_1$ 
\item $|\{f \leq 0\} \cap Q^-| \geq \mu$.
\item $\|G\|_{L^\infty(Q_1)} \leq \eps_0$.
\end{itemize}
Then \[ \esssup_{Q_{1/2}} f \leq 1-\theta,\]
for some $\theta>0$ depending on $\mu$, the ellipticity parameters of the coefficients, $\|b\|_{L^\infty}$, and dimension only.
\end{lemma}

\section{The upper bound for subsolutions}

In this short section we give the proof of Theorem \ref{t:Linfty}. As we will see, it is a relatively quick consequence of Lemma \ref{l:extending-sub-as-zero} together with the interior upper bound of Proposition \ref{p:interior-Linfty}.

\begin{proof}[Proof of Theorem \ref{t:Linfty}]

We start by extending $f$ to all of $Q_1(\z)$ as zero.
\[ \bar f(z) := \begin{cases}
f(z) & \text{if } z \in \HH_1(\z), \\
0 & \text{if } z \in Q_1(\z) \setminus \HH_1(\z). \end{cases}
\]

According to Lemma \ref{l:extending-sub-as-zero}, the function $\bar f$ is a subsolution of \eqref{e:kFP} in the weak sense in $Q_1(\z)$. We can then apply Proposition \ref{p:interior-Linfty} (translated by $z_0 \circ$) and get
\[ \esssup_{Q_{1/2}(\z)} \bar f \leq C \left( \|\bar f\|_{L^2(Q_1(\z))} + \|G\|_{L^\infty(Q_1(\z))} \right).\]
Observing that $\esssup_{Q_{1/2}(\z)} \bar f = \esssup_{\HH_{1/2}(\z)} f$ and that $\|\bar f\|_{L^2(Q_1(\z))} = \|f\|_{L^2(\HH_1(\z))}$, we conclude the proof.
\end{proof}

\section{H\"older continuity}

We move to the proof of Theorem \ref{t:Calpha}. We set up an improvement of oscillation scheme as is common for H\"older continuity proofs. The way the improvement of oscillation works differs depending on whether we analyze the solution away from the incoming boundary $\gamma_-$, or near it. We start by analyzing the solution away from $\gamma_-$.

\begin{lemma} \label{l:oscillation-away-from-influx}
Let $f: \HH_1(\z) \to \R$ be a bounded weak solution of the equation \eqref{e:kFP}. Assume there is no incoming boundary inside of $Q_1(\z)$, so $f$ does not satisfy any boundary condition. There exists $\theta>0$ and $\eps_0>0$, depending on the ellipticity parameters, $\|b\|_{L^\infty}$ and dimension only, so that if $\osc_{\HH_1(\z)} f \leq 1$ and $\|G\|_{L^\infty} \leq \eps_0$, then
\[ \osc_{\HH_1(\z)} f \leq 1-\theta. \]
\end{lemma}

\begin{proof}

Let $m = (\esssup_{\HH_1(\z)} f + \essinf_{\HH_1(\z)} f)/2$.

Recall the set $Q^-$ defined in section \ref{s:interior}. Let us consider the following two sets.

\begin{align*} 
Q_1^- &:= \{z \in Q^- : z \notin \HH_1(\z) \text{ or } f(z) \leq m \}, \\
Q_2^- &:= \{z \in Q^- : z \notin \HH_1(\z) \text{ or } f(z) \geq m \}.
\end{align*}

Since the union of these two sets equals $Q^-$, one of them must have measure $\geq |Q^-|/2$. Let us suppose it is $Q_1^-$. The case it is $Q_2^-$ is analyzed identically substituting $f$ with $-f$.

Consider the function $\tilde f = 2 (f-m)_+$. Since $\osc f \leq 1$, we must have $0 \leq \tilde f(z) \leq 1$ for all $z \in \HH_1(\z)$. Since there is no incoming boundary inside $Q_1$
we apply Corollary \ref{c:convex-sub-no-influx} to conclude that $\tilde f$ is a nonnegative subsolution with $G$ replaced by $G \one_{f>m}$. 

We extend $\tilde f$ as zero in the full cylinder $Q_1(\z)$ applying Lemma \ref{l:extending-sub-as-zero}. Note that $\tilde f(z)=0$ at any point so that either $f(z) \leq m$, or $x \notin \Omega$. This set has measure at least $|Q^-|/2$. We apply Lemma \ref{l:growth-lemma} to get that $\tilde f \leq 1 - \tilde \theta$ in $Q_{1/2}(\z)$ for some $\tilde \theta > 0$.

Going back to the original function $f$, we got that $f \leq (1-\tilde \theta)/2 + m$ in $\HH_{1/2}(\z)$, and therefore $\osc_{\HH_{1/2}(\z)} f \leq (1-\theta/2)$.
\end{proof}

By the standard iteration of rescalings of Lemma \ref{l:oscillation-away-from-influx}, we conclude that the solution $f$ must be H\"older continuous away from the incoming boundary $\gamma_-$.

\begin{corollary} \label{c:Holder-away-from-influx}
Let $f: \HH_1(\z) \to \R$ be a bounded weak solution of the equation \eqref{e:kFP}. Assume there is no incoming boundary inside of $Q_1(\z)$, so $f$ does not satisfy any boundary condition. Then
\[ \|f\|_{C^\alpha(\HH_{1/2}(\z))} \leq C \left( \|f\|_{L^\infty(\HH_1(\z))} + \|G\|_{L^\infty(\HH_1(\z))} \right), \]
for some constants $\alpha>0$ and $C$ depending on the ellipticity parameters, $\|b\|_{L^\infty}$ and dimension only.
\end{corollary}

The kinetic cylinders $Q_r(\z)$ (defined in the introduction) are oblique with respect to the spatial variable $x$. The way they intersect the domain $\{x \in \Omega\}$ is different when $\z \in \gamma_+$ or $\z \in \gamma_-$. If $\z \in \gamma_0$, then roughly half of $Q_r(\z)$ will be inside the domain $\{x \in \Omega\}$ and the other half outside. If $\z \in \gamma_+$, then $\vz \cdot n > 0$, so the cylinder $Q_r(\z)$ flows from inside $\Omega$. Recall that $\tz$ is the final time of $Q_r(\z)$. A larger proportion of $Q_r(\z)$ will be inside the domain when $\z \in \gamma_+$. Conversely, when $\z \in \gamma_-$, then $Q_r(\z)$ flows from the outside, and therefore a larger proportion of $Q_r(\z)$ will not belong to $\Omega$. See Figure \ref{f:in-vs-out}.

\begin{center}
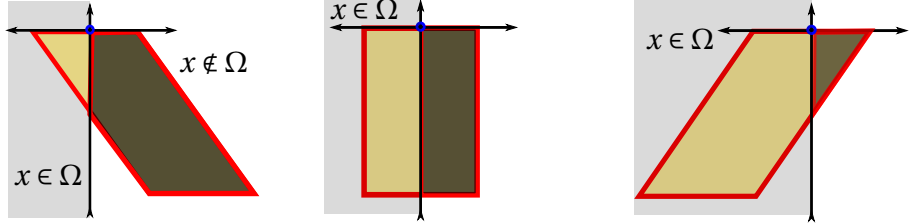
\begin{figure}[ht!] \caption{These drawings represent the intersection of $Q_1(\z)$ with $\{x \in \Omega\}$. The horizontal variable is $x$ (one-dimensional) and the vertical variable is $t$. The variable $v$ is not drawn. The points to the left of the vertical axis represent $\{x \in \Omega\}$. The picture on the left is for the case $\z \in \gamma_-$, the center picture is for $\z \in \gamma_0$, and the picture on the right is for $\z \in \gamma_+$.} \label{f:in-vs-out} 
\setlength{\unitlength}{1cm} 
\begin{picture}(12,3)
\put(0.1,0.5){$x \in \Omega$}
\put(2.3,2){$x \notin \Omega$}
\put(4.3,2.7){$x \in \Omega$}
\put(8.5,2.3){$x \in \Omega$}
\includegraphics[width=12cm,height=3cm]{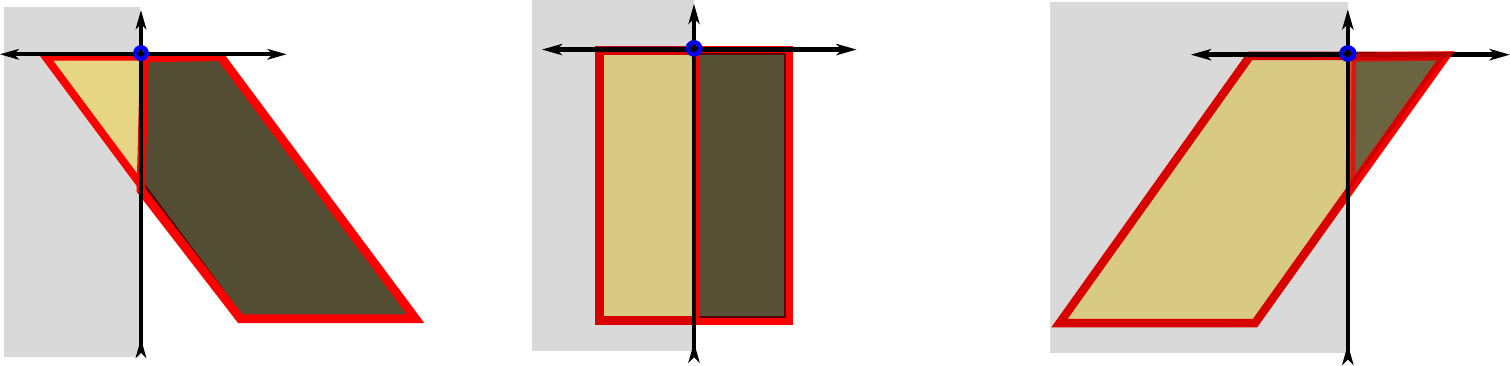}
\end{picture}
\end{figure}
\end{center}

Let us now consider the case in which the origin is close to some point on $\gamma_-$. The following lemma allows us to take advantage of the convexity assumption. It says that if there is any point in $\gamma_-$ sufficiently close to $\z$, then a large portion of $Q^-(\z)$ must be outside of the spatial domain $\Omega$.

\begin{lemma}\label{l:convex-and-Q-}
Assume that $\Omega$ is convex and that there is some point in $\gamma_- \cap Q_{1/8}(\z)$. Then, there is a $\mu > 0$ depending on dimension only so that
\[ | Q^-(\z) \cap \{x \notin \Omega\}| > \mu.\]
\end{lemma}

\begin{proof}
Let $z^1 \in \gamma_- \cap Q_{1/8}(\z)$. We write $z^1 = (t^1, x^1, v^1)$. Since $z^1 \in \gamma_-$, we know that $v^1 \cdot n < 0$, where $n$ is the unit normal vector to $\partial \Omega$ at $x^1$. Since $z^1 \in Q_{1/8}(\z)$, we have in particular that $\vz \cdot n = (\vz-v^1)\cdot n + v^1 \cdot n < 1/8$, where $\z  = (\tz, \xz, \vz)$.

We are assuming that $\Omega$ is convex. Therefore, we have
\[ \{ x_0+y \in \R^d: y \cdot n > 0 \} \subset \R^d \setminus \Omega.\]

Recall that
\[ Q^- = (-3/4,-1/2] \times B_{1/2^3} \times B_{1/2},\]
and thus
\[ \begin{aligned}
Q^-(\z) = \{ (t,x,v) : & t \in (\tz-3/4,\tz-1/2], \\
&x-\xz - (t-\tz) \vz \in B_{1/8}, \\
&v - \vz \in B_{1/2} \} 
\end{aligned} \]

We know $x \notin \Omega$ when $(x-\xz) \cdot n > 0$. Therefore
\[ \begin{aligned}
Q^-(\z) \cap \{x \notin \Omega\} \supset \{ (t,x,v) : & t-\tz \in (-3/4,-1/2], \\
&x-\xz - (t-\tz) \vz \in B_{1/8} \text{ and } (x-\xz) \cdot n > 0 , \\
&v - \vz \in B_{1/2} \} 
\end{aligned} \]
Since $\vz \cdot n < 1/8$ and $t-\tz \in  (-3/4,-1/2]$, $-(t-\tz) \vz \cdot n < 3/32$. Therefore, the conditions $x-\xz - (t-\tz) \vz \in B_1$ and $(x-\xz) \cdot n > 0$ will certainly hold when $\tilde x = x-\xz - (t-\tz) \vz \in B_{1/8}$ and $\tilde x \cdot n > 3/32$. Since $3/32 < 1/8$, these values of $\tilde x$ comprise a nonempty subset of the ball $B_{1/8}$. We conclude the proof of the lemma setting
\[ \mu = (3/4-1/2) \times |B_1 \cup \{\tilde x : \tilde x \cdot n > 3/32 \}| \times |B_1|. \]
\end{proof}

\begin{lemma} \label{l:oscillation-near-influx}
Let $f: \HH_1(\z) \to \R$ be a bounded weak solution of the equation \eqref{e:kFP} with the boundary condition \eqref{e:influx}. Assume that $\Omega$ is convex and that there is some point in $\gamma_- \cap Q_{1/8}(\z)$. Let $m = \sup_{\gamma_- \cap Q_1(\z)} g$. There exists $\theta>0$ and $\eps_0>0$, depending on the ellipticity parameters, $\|b\|_{L^\infty}$ and dimension only, so that if $f-m \leq 1$ in $\HH_1(\z)$ and $\|G\|_{L^\infty} \leq \eps_0$, then
\[ f-m \leq 1-\theta \text{ in } \HH_{1/2}(\z).\]
\end{lemma}

\begin{proof}
Using Lemma \ref{l:convex-and-Q-}, we know that the set $Q^-(\z) = \z \circ Q^-$ intersects $\{(t,x,v): x \notin \Omega\}$ in a set of measure at least $\mu$, for some $\mu > 0$ depending on dimension. 

Consider the function $\tilde f = (f-m)_+$. Since $f-m \leq 1$, we must have $0 \leq \tilde f(z) \leq 1$ for all $z \in \HH_1(\z)$. We apply Lemma \ref{l:convex-sub} to conclude that $\tilde f$ is a subsolution with $G$ replaced by $G \one_{f>0}$. 

We extend $\tilde f$ as zero in the full cylinder $Q_1(\z)$ applying Lemma \ref{l:extending-sub-as-zero}. Note that $\tilde f(z) = 0$ at every point $z(t,x,v)$ so that $x \notin \Omega$. In particular, such points in $Q^-(\z)$ have measure at least $\mu > 0$ because of Lemma \ref{l:convex-and-Q-}.

We apply Lemma \ref{l:growth-lemma} to get that $\tilde f \leq 1 - \tilde \theta$ in $Q_{1/2}(\z)$ for some $\tilde \theta > 0$, and finish the proof.
\end{proof}

We are now ready to write the proof of Theorem \ref{t:Calpha}. It follows by iteration of Lemmas \ref{l:oscillation-away-from-influx} and \ref{l:oscillation-near-influx}. We write the details below.

\begin{proof}[Proof of Theorem \ref{t:Calpha}]
Let us prove the H\"older continuity at the point $\z$. There is no difficulty to translate this proof to any other point in $\HH_{1/2}(\z)$. We assume $\xz \in \overline \Omega$. Note that $\z$ may or may not belong to the boundary $\gamma$.

We prove the convex case only. If $\Omega$ is not convex, we flatten the boundary with the procedure described in Section \ref{s:flattening} and reduce it to the convex case.

By scaling if necessary, we can assume that $\|b\|_{L^\infty} \leq 1$. Dividing the function $f$ by $\|f\|_{L^2(\HH_1(z^0))} + \|g\|_{C^\alpha(\gamma_- \cap Q_1(\z))}/\eps_0 + \|G\|_{L^\infty(\HH_1(\z))}/\eps_0$, we can and do assume that 
\begin{align*}
\|f\|_{L^2(\HH_1(z^0))} &\leq 1, \\
\|g\|_{C^\alpha(\gamma_- \cap Q_1(\z))} &\leq \eps_0, \\
\|G\|_{L^\infty(\HH_1(\z)} &\leq \eps_0.
\end{align*}

We pick $\alpha>0$ to be the number so that $2^{-\alpha} = (1-\theta/2)$, where $\theta>0$ is the minimum between the two positive constants in Lemmas \ref{l:oscillation-away-from-influx} and \ref{l:oscillation-near-influx}.

Under these assumptions, we want to prove that there is a constant $C$, depending on the ellipticity parameters, $\|b\|_{L^\infty}$ and dimension, so that for any $r \in (0,1/2)$, we have $\osc_{\HH_{r}(\z)} f \leq C r^\alpha$.

Applying Theorem \ref{t:Linfty} to $(f-\sup_{\gamma_- \cap Q_1(\z)} g)_+$ and $(f-\inf_{\gamma_- \cap Q_1(\z)} g)_-$, we deduce that
\[ \osc_{\HH_{1/2}} f \leq C_0, \]
for some constant $C_0$.

We now iterate either Lemma \ref{l:oscillation-near-influx} or Corollary \ref{c:Holder-away-from-influx} to estimate the oscillation of $f$ in $\HH_{2^{-k}}$ for $k=1,2,3,\dots$.

We claim that for as long as $\gamma_- \cap Q_{2^{-k-3}}(\z) \neq \emptyset$, we have
\begin{equation} \label{pf:o1}
\osc_{\HH_{2^{-k}}} f \leq C_0 (1-\theta/2)^{k-1}.
\end{equation}
We proceed to prove this inequality iterating Lemma \ref{l:oscillation-near-influx}.

Indeed, we already know this inequality holds for $k=1$. We prove it by induction for larger values of $k$. The inductive hypothesis tells us it holds for some given $k \geq 1$ and we assume that $\gamma_- \cap Q_{2^{-k-3}}(\z) \neq \emptyset$. Let $m_0 = \essinf_{\gamma_- \cap Q_{2^{-k}}(\z)} g$ and $m_1 = \esssup_{\gamma_- \cap Q_{2^{-k}}(\z)} g$. Since $\|g\|_{C^\alpha} \leq \eps_0$, we have $0 \leq m_1 - m_0 \leq \eps_0 2^{-k \alpha}$.

Let $M_0 = \esssup_{\HH_1(\z_0)} (f-m_0)_+$ and $M_1 = \esssup_{\HH_1(\z_0)} (m_1-f)_+$. We note that, using our inductive hypothesis,
\[ \osc_{\HH_{2^{-k}}} f = M_1 + M_0 + (m_1-m_0) \leq C_0 (1-\theta/2)^{k-1}. \]
We apply Lemma \ref{l:oscillation-near-influx} to $(f-m_0)/M_0$ and $(m_1-f)/M_1$ rescaled. We obtain that
\begin{align*}
\esssup_{\HH_{2^{-k-1}}} (f-m_0) \leq (1-\theta) M_0, \\
\esssup_{\HH_{2^{-k-1}}} (m_1-f) \leq (1-\theta) M_1.
\end{align*}
Therefore,
\[ \osc_{\HH_{2^{-k-1}}} f \leq (1-\theta)(M_0+M_1) + (m_1-m_0). \]
Using the inductive hypothesis and our bound for $m_1-m_0$,
\[ \osc_{\HH_{2^{-k-1}}} \leq (1-\theta)C_0 (1-\theta/2)^k + \eps_0 2^{-k \alpha} \leq (1-\theta)C_0 (1-\theta/2)^k + \eps_1 (1-\theta/2)^{k+1}. \]
The last inequality holds provided that $2^{-\alpha} \leq (1-\theta/2)$ and $\eps_1$ is some small number depending on $\eps_0$. Prodided that $\eps_0$ (and therefore also $\eps_1$) is sufficiently small, we conclude
\[ \osc_{\HH_{2^{-k-1}}} \leq C_0 (1-\theta/2)^{k+1}. \]
Thus, we proved \eqref{pf:o1} for as long as $\gamma \cap Q_{2^{-k-3}}$ is not empty. Eventually, there might be some value of $k=k_0$ so that $\gamma \cap Q_{2^{-k_0-3}}$ is empty. In that case, we trivially know that if we set $C_1 = (1-\theta)^{-4} C_0$, then the inequality
\begin{equation} \label{pf:o2}
\osc_{\HH_{2^{-k}}} f \leq C_1 (1-\theta/2)^{k-1}.
\end{equation}
holds up to $k=k_0+3$. From this point on, we iterate Lemma \ref{l:oscillation-away-from-influx} and conclude that \eqref{pf:o2} holds for all values of $k=1,2,3,\dots$. It is a standard fact that the H\"older continuity estimate follows from \eqref{pf:o2}. 
\end{proof}

\section{Vanishing of infinite order on the incoming boundary}
\label{s:Cinfty}

The key to understand why Theorem \ref{t:Cinfty} holds is that when $\z \in \gamma_-$, only a tiny proportion of $Q_r(\z)$ intersects $\{x \in \Omega\}$ for small $r>0$. Indeed, as pictured in Figure \ref{f:in-vs-out}, kinetic cylinders centered on the incoming part of the boundary consist mostly of points that are outside of the domain $\Omega$. As a consequence of the kinetic scaling, this effect is enhanced for small values of $r$. One could even argue that every point $\z \in \gamma_-$ is cusp-like in the sense that
\[ \limsup_{r \to 0} \frac{|Q_r(\z) \cap \{x \in \Omega\}|}{|Q_r(\z)|} = 0 .\]

The next lemma takes advantage of the intuition above. Theorem \ref{t:Cinfty} will follow as a consequence.

\begin{lemma} \label{l:high-decay-incoming}
Let $f: \HH_r(\z) \to \R$ be a bounded weak solution of the equation \eqref{e:kFP} with the boundary condition \eqref{e:influx}. Assume that $\Omega$ is convex and $\z \in \gamma_-$. For any $\delta>0$, there exist an $r_0 > 0$, depending on $\delta$ and $\vz \cdot n$ only, and $\eps_0>0$, depending on $\delta$ only, so that the following statement is true. If $g \leq m$ in $Q_r(\z) \cap \gamma_-$ for some $r \leq r_0$, then 
\[ \|(f-m)_+\|_{L^\infty(\HH_{r/2}(\z))} \leq \delta \|(f-m)_+\|_{L^\infty(\HH_r(\z))} + Cr^2 \|G\|_{L^\infty(\HH_r(\z))}.\]
\end{lemma}

\begin{proof}
The key of this lemma is to estimate the volume of the intersection of $\HH_r(\z)$. We claim that for $r < r_0$ (with depending on $\vz \cdot n$ and the distance between $\z$ and $\gamma_+$), the volume of $\HH_r(\z)$ is an arbitrarily small fraction of the volume of $Q_r(\z)$. Once we establish that, the result follows by applying Theorem \ref{t:Linfty} to $(f-m)_+$. Indeed, $(f-m)_+$ is a nonnegative subsolution in $\HH_r(\z)$ from Lemma \ref{l:convex-sub}. Applying Theorem \ref{t:Linfty} properly scaled, we get
\begin{align*}
\esssup_{\HH_{r/2}(\z)} (f-m)_+ &\leq C |Q_r|^{-1/2} \|(f-m)_+\|_{L^2(\HH_r(\z))} + Cr^2 \|G\|_{L^\infty(\HH_r(\z))}, \\
&\leq C |Q_r|^{-1/2} |\HH_r(\z)|^{1/2} \|(f-m)_+\|_{L^\infty(\HH_r(\z))} + Cr^2 \|G\|_{L^\infty(\HH_r(\z))}, \\
&\leq \delta \|(f-m)_+\|_{L^\infty(\HH_r(\z))} + Cr^2 \|G\|_{L^\infty(\HH_r(\z))}.
\end{align*}
where $\delta := C |Q_r|^{-1/2} |\HH_r(\z)|^{1/2}$ is small provided that $r < r_0$ is small.

Therefore, we are left to prove that we can make $|\HH_r(\z)| / |Q_r|$ arbitrarily small by choosing $r < r_0$ small.

Since $\Omega$ is convex, we know that the set
\[ \{ \xz + y \in \R^d : y \cdot n \geq 0 \} \]
does not intersect $\Omega$.

The map $(t,x,v) \mapsto \z \circ (r^2 t, r^3 x, rv)$ maps $Q_1$ to $Q_r(\z)$. Let $\OO_r(\z)$ be the pre-image by this map of $\HH_r(\z)$. We observe that $|\HH_r(\z)| / |Q_r| = |\OO_r(\z)| / |Q_1|$.

The set $\OO_r(\z)$ can be expressed explicitly by the formula
\[ |\{ (t,x,v) \in Q_1 : \xz + r x + t \vz \in \Omega \}| < \delta/C. \]

Since $\Omega$ is convex, we know that
\[ \Omega \subset \{ \xz + y \in \R^d : y \cdot n < 0 \}. \]
We assume that $\z \in \gamma_-$, thus $\vz \cdot n<0$. Recall that also $t<0$ in $Q_1$. Thus, if $r < r_0$,
\[ \OO_r(\z) \subset \{ (t,x,v) \in Q_1 : r_0 x \cdot n < -t \vz \cdot n \}. \]
For any given $t<0$, the quantity $-t \vz \cdot n$ is negative. The set on the right-hand side has an arbitrarily small measure provided that we pick $r_0$ small (depending on $v_0 \cdot n$ only). We pick $r_0>0$ so that its measure is sufficiently small and we finish the proof.
\end{proof}

The proof of Theorem \ref{t:Cinfty} follows by iterating Lemma \ref{l:high-decay-incoming} with $m=0$ and $G=0$.

\begin{proof}[Proof of Theorem \ref{t:Cinfty}]
We prove the convex case only. If $\Omega$ is not convex, we flatten the boundary with the procedure described in Section \ref{s:flattening} and reduce it to the convex case.

For any $q > 0$, let us pick $\delta < 1/2^q$. Let $r_0 > 0$ be the positive radius of Lemma \ref{l:high-decay-incoming}. 

Applying Lemma \ref{l:high-decay-incoming} iteratively to $f$ and $-f$ with the radii $r_0, r_0/2, r_0/4, r_0/8, \dots$, we obtain
\[ \|f\|_{L^\infty(\HH_{2^{-k} r_0}(\z))} \leq \delta^k \|f\|_{L^\infty(\HH_{r_0}(\z))}.\]
Therefore, for $r < r_0$ and $C = \delta^{-1}$, we have
\[ \|f\|_{L^\infty(\HH_{r}(\z))} \leq C r^q \|f\|_{L^\infty(\HH_{r_0}(\z))} \leq C r^q \|f\|_{L^\infty(\HH_{1}(\z))}.\]
\end{proof}

It is possible to modify Theorem \ref{t:Cinfty} to allow for nonzero boundary conditions and right-hand side. However, the decay of the function $f$ on the boundary will be limited by the scaling of the boundary data and right-hand side. There are several variants that one may get depending on the assumptions. In every case, we compute the regularity of $f$ on $\gamma_-$ by iterating Lemma \ref{l:high-decay-incoming}.

Let us analyze the case $G \neq 0$ but $g=0$. In this case, the direct iteration of Lemma \ref{l:high-decay-incoming} gives us
\begin{align*} 
\|f\|_{L^\infty(\HH_{2^{-k} r_0}(\z))} &\leq \delta^k \|f\|_{L^\infty(\HH_{r_0}(\z))} + \left( \delta^{k-1} r_0^2  + \delta^{k-2} (2^{-1} r_0)^2 + \dots + (2^{-k} r_0)^2 \right) \|G\|_{L^\infty} ,\\
&\leq \delta^k \|f\|_{L^\infty(\HH_{r_0}(\z))} + \left( (2\delta)^{k-1}  + (2\delta)^{k-2} + \dots + (2\delta) +  1 \right) (2^{-k} r_0)^2 \|G\|_{L^\infty} ,\\
&\leq \delta^k \|f\|_{L^\infty(\HH_{r_0}(\z))} + (1-2\delta)^{-1} (2^{-k} r_0)^2 \|G\|_{L^\infty}.
\end{align*}

Thus, for $r < r_0$, we get
\begin{equation} \label{e:cylinder-decay-on-incoming-2} 
\|f\|_{L^\infty(\HH_{r}(\z))} \leq C r^q \|f\|_{L^\infty(\HH_{r_0}(\z))} + C r^2 \|G\|_{L^\infty}.
\end{equation}
Observe that no matter how small we pick $\delta>0$, we do not get better than a quadratic exponent in the second term.\footnote{We would get a slightly better exponent if we replaced Proposition \ref{p:interior-Linfty} with a sharper version that takes into account the $L^p$ norm of $G$, instead of its $L^\infty$ norm. An interior estimate of that kind is considered in \cite{anceschi2021note}, in a more general context.}

If we also have a nonzero boundary condition $g \in C^\alpha$ (for any given $\alpha \in (0,1)$), following the same iteration of Lemma \ref{l:high-decay-incoming} we get
\begin{equation} \label{e:cylinder-decay-on-incoming}
\osc_{\HH_r(\z)} f \leq C r^q \|f\|_{L^\infty(\HH_{r_0}(\z))} + C r^2 \|G\|_{L^\infty} + C r^\alpha \|g\|_{C^\alpha},
\end{equation}
for all $r < r_0$. In this case, the dominant term for $r$ small will be the third one, that has exponent $r^\alpha$.

It is interesting to restate \eqref{e:cylinder-decay-on-incoming} in terms of the distance of any arbitrary point $z$ to the incoming boundary $\gamma_- \cap \{v \cdot n > \nu_0\}$. One would naively guess that the right-hand side in \eqref{e:cylinder-decay-on-incoming} corresponds to an estimate in terms of this distance to the power $\alpha$. However, we obtain a slightly smaller exponent for a reason that will be explained below.

Let us analyze the case $g=0$. Consider $z$ to be any point in the domain of the equation. In order to get the best possible upper bound for $f(z)$ using \eqref{e:cylinder-decay-on-incoming-2}, we should look for a point $\z \in \gamma_-$ so that $z \in \HH_r(\z)$ and $r>0$ is the smallest possible value. Note that, for small values of $r$, there is only a small proportion of the cylinder $Q_r(\z)$ inside $\HH_r(\z)$ (see the proof of Lemma \ref{l:high-decay-incoming}). The points of $\HH_r(\z)$ concentrate near $t=t_0$ for small $r$. We should therefore take $\tz = t$, and pick $(\xz, \vz)$ on the incoming part of the boundary and closest to $(x,v)$. Thus, our optimal choice is to take $(\xz, \vz)$ so that
\[ \max( |\xz-x|^{1/3}, |\vz-v|) = \min\{\max( |y-x|^{1/3}, |w-v|) : y \in \partial \Omega, w \in \R^d : n \cdot w < 0\}. \]
Nota that the set on the right-hand side is not closed. The infimum may be achieved on the boundary where $n \cdot w = 0$, so $\z \in \gamma_0$. Equation \eqref{e:cylinder-decay-on-incoming-2} does not apply and the best estimate we have in this case is the H\"older modulus of continuity given by Theorem \ref{t:Calpha}.

If we get $\z \in \gamma_0$ for the value of $\z$ defined above, we have $z \in \HH_r(\z)$ for $r = \max( |\xz-x|^{1/3}, |\vz-v|)$. We can then apply \eqref{e:cylinder-decay-on-incoming-2} and get
\[ \|f\|_{L^\infty(\HH_{r}(\z))} \leq C \max( |\xz-x|^{1/3}, |\vz-v|)^m \|f\|_{L^\infty} + C \max( |\xz-x|^{1/3}, |\vz-v|)^2 \|G\|_{L^\infty}.
\]
The constant $C$ depends on $\vz \cdot n$.

It is easy to verify that $\max( |\xz-x|^{1/3}, |\vz-v|) \leq d(z,\gamma_-)^{2/3}$ directly from Definition \ref{d:distance}. The exponent $2/3$ cannot be improved in general, as we can see in the following one-dimensional example: if $\Omega = \{ x < 0 \}$ and we take $(t,x,v) = (r^2,-r^2,-1)$, we have $d(z,\gamma_-) = d(z,(0,0,1)) = r$. However, the point $(r^2,-r^2,1)$ does not belong to $Q_r(0,0,1)$ (we instead have $(0,0,1) \in Q_r(z)$). In this example $\z = (r^2,0,1)$ and $\max( |\xz-x|^{1/3}, |\vz-v|) = r^{2/3}$.

%%      ---------------------------------------------------------------------
%%      --------------------------- BIBLIOGRAPHY ----------------------------
%%      ---------------------------------------------------------------------
%% PUT HERE THE BIBLIOGRAPHY IN YOUR FAVOURITE FORMAT
%% Please check that the format of the bibliography is uniform and coherent

\bibliography{kfpboundary}
\bibliographystyle{plain}
\end{document}